\documentclass[reqno,12pt]{amsart}
\usepackage{vmargin}
\setpapersize{A4}

\usepackage{amsmath}

\usepackage{amsfonts}

\usepackage{amssymb}

\usepackage{eufrak}

\usepackage{amscd}

\usepackage{amsthm}

\usepackage{epsfig}

\usepackage{amstext}

\usepackage[all,line,dvips]{xy}
\CompileMatrices 

\newcommand{\id}{\operatorname{id}}

 \newcommand{\Ext}{\operatorname{Ext}}

 \newcommand{\ev}{\operatorname{ev}}

   \theoremstyle{plain}
   \newtheorem{thm}{Theorem}[section]
   \newtheorem{prop}[thm]{Proposition}
   \newtheorem{lemma}[thm]{Lemma}
   \newtheorem{cor}[thm]{Corollary}
   \theoremstyle{definition}
   
   \newtheorem{defn}[thm]{Definition}
   
   \theoremstyle{remark}

\usepackage{enumerate}
\usepackage{graphicx}
\usepackage{enumitem}

   \numberwithin{equation}{section}

\makeatletter
        \date{\today}

\title[Shape theory and extensions of $C^*$-algebras]{Shape theory and
  extensions of $C^*$-algebras}
\author{Vladimir Manuilov and Klaus Thomsen}

\date{}

\email{matkt@imf.au.dk}

  \address{Dept. of Mech. and Math.\\
Moscow State University\\
Moscow, 119991, Russia}
\address{Institut for matematiske fag, Ny Munkegade, 8000 Aarhus C,
  Denmark}

\begin{document}

\maketitle

\begin{abstract}
Let $A$, $A'$ be separable $C^*$-algebras, $B$ a stable
$\sigma$-unital $C^*$-algebra. Our main result is the construction
of the pairing
$[[A',A]]\times\operatorname{Ext}^{-1/2}(A,B)\to\operatorname{Ext}^{-1/2}(A',B)$,
where $[[A',A]]$ denotes the set of homotopy classes of asymptotic
homomorphisms from $A'$ to $A$ and
$\operatorname{Ext}^{-1/2}(A,B)$ is the group of semi-invertible
extensions of $A$ by $B$. Assume that all extensions of $A$ by $B$
are semi-invertible. Then this pairing allows us to give a
condition on $A'$ that provides semi-invertibility of all
extensions of $A'$ by $B$. This holds, in particular, if $A$ and
$A'$ are shape equivalent. A similar condition implies that if
$\operatorname{Ext}^{-1/2}$ coincides with $E$-theory (via the
Connes-Higson map) for $A$ then the same holds for $A'$.

\end{abstract}

\section{Introduction}

The theory of
extensions of $C^*$-algebras is presently experiencing an
unprecedented level of activity aiming to
improve our understanding
of extensions of non-nuclear $C^*$-algebras. One line of
research was sparked by the examples of non-invertible extensions by
the reduced group $C^*$-algebra of a free group obtained by Haagerup
and Thorbj\o rnsen in \cite{HT} and the subsequent applications of their
result by Hadwin and Shen in \cite{HS} which has resulted in a wealth
of examples of $C^*$-algebras with non-invertible extensions by the
compact operators $\mathbb K$. As pointed out in \cite{MT5} each such
example gives rise to a non-invertible extension of the same
$C^*$-algebra by any $C^*$-algebra of the form $B \otimes \mathbb K$
with $B$ unital. Although it may still be a pre-mature to
conclude that the presence of non-invertible extensions is more of a
rule than an exception in the non-nuclear case, the new wealth of
examples has made it more urgent to find a way to handle
non-invertible extensions.

In our previous work we have proposed an approach
towards an analysis of $C^*$-extensions in which many
non-invertible extensions can be handled in a way analogous to how
invertible extensions are dealt with and classified in the theories
developed by Brown, Douglas and Fillmore, \cite{BDF}, and Kasparov, \cite{K1}. Specifically, in a series of papers, beginning with
 \cite{MT3} and culminating in \cite{MT5}, it has been shown that many of the non-invertible
extensions are invertible in a slightly weaker sense, called
\emph{semi-invertibility}. Recall that an extension of a
$C^*$-algebra $A$ by a stable $C^*$-algebra $B$ is invertible when
there is another extension, the inverse, with the property that the
direct sum extension of the two is a split
extension. Semi-invertibility requires only that the sum is
\emph{asymptotically split}, in the sense that there is an asymptotic
homomorphism as defined by Connes and Higson, \cite{CH}, consisting
of right-inverses of the quotient map. What has been shown is that many classes of (non-nuclear) $C^*$-algebras, including suspensions,
certain full and reduced group $C^*$-algebras and certain amalgamated
free products, have the property that all extensions of the algebra by a
stable ($\sigma$-unital) $C^*$-algebra are semi-invertible. For some
of these algebras it is known, thanks to the development mentioned
above, that there exist non-invertible extensions, but for many or
most it is simply not known if all extensions are invertible or
not. Intriguingly it has also been shown, in \cite{MT4}, that
non-semi-invertible extensions exist.

The main reason why semi-invertibility is easier to establish, and a
good reason why it can appear to be more natural in the homology and
co-homology theories that are based on extensions of $C^*$-algebras is
that it is homotopy invariant, in the sense that if the Busby
invariant of two $C^*$-extensions are homotopic as $*$-homomorphisms
then one of the extensions is semi-invertible if and only if the other
is. This is in glaring contrast to invertibility; it is e.g. known
that there are contractible $C^*$-algebras with non-invertible
extensions by $\mathbb K$, cf. \cite{Ki}. All but one of the methods
used so far to establish automatic semi-invertibility use the homotopy
invariance property; the
exception being Theorem 3.3 of \cite{MT5}. However, it is shown in
\cite{MT4} that there are $C^*$-extensions which are not even
invertible up to the more natural and much weaker notion of homotopy
usually applied in connection with $C^*$-extensions, and it becomes
therefore a natural problem to identify the borderline between the
$C^*$-algebras for which semi-invertibility of extensions is
automatic, and the rather mysterious algebras with non-semi-invertible
extensions. Presently such an identification seems out of reach,
although one may be slightly more optimistic about the possibility of
finding the right separating conditions than for doing the analogous
thing concerning invertibility.

The main purpose with the present
paper is to show that automatic semi-invertibility of extensions, as a
property of $C^*$-algebras, is
not only invariant under homotopy equivalence, but also
under shape equivalence. This allows us to identify a large natural
class of $C^*$-algebras which have this property, namely the class of $C^*$-algebras
whose shape is dominated by a nuclear $C^*$-algebra. To make this
more precise recall that shape theory of $C^*$-algebras was introduced
by Effros and Kaminker in \cite{EK} as a generalisation of shape
theory for topological spaces. It was developed further by Blackadar
in \cite{B} before it was tied together with the E-theory of Connes
and Higson by Dadarlat in \cite{D}. Roughly speaking what Dadarlat
showed was that shape theory of $C^*$-algebras can be described by the
homotopy category of asymptotic homomorphisms which suitably suspended
becomes the E-theory of Connes and Higson, \cite{CH}. In short, shape
theory is unsuspended E-theory. It is in this guise that we use shape
theory here. As shown by Dadarlat a morphism in the shape category, say from the
$C^*$-algebra $A$ to the $C^*$-algebra $B$, is given by an element in
$[[A,B]]$, the homotopy classes of asymptotic homomorphism from $A$ to
$B$. Our main result says that if $A$ has the property that there is another
$C^*$-algebra $A'$ and asymptotic homomorphisms $\psi : A \to
A \otimes \mathbb K$, $\lambda : A \to A'\otimes \mathbb K$ and $\mu : A'\otimes
\mathbb K \to A\otimes \mathbb K$ such that
 $$
\left[\id_A\right] + [\psi] = [\mu] \bullet [\lambda]
$$
in $[[A,A\otimes \mathbb K]]$, where $\id_A$ is the identity map on
$A$, considered as a map $A \to A \otimes \mathbb K$ in the standard
way and $\bullet$ denotes the composition product of Connes and
Higson, then all extensions of $A$ by a stable $\sigma$-unital
$C^*$-algebra $B$ are semi-invertible if all extensions of $A'$ by $B$
are. When $\psi$ can be taken to be zero the assumption means
that $A$ is shape dominated by $A'$ in a sense generalising
the notion of homotopy domination introduced by Voiculescu, \cite{V},
and when $A'$ can be taken
to be zero the assumption is that $A$ is homotopy symmetric in the sense
defined by Dadarlat and Loring in \cite{DL}.

We consider also the relation between the group of semi-invertible
extensions and E-theory
proper. As we showed in \cite{MT4} the Connes-Higson construction introduced in \cite{CH} does
not give an isomorphism between E-theory and the homotopy classes of
extensions in general, but we show that it does for $C^*$-algebras
that are homotopy symmetric or shape dominated by a nuclear
$C^*$-algebra, or any other $C^*$-algebra for which it does.

\emph{Acknowledgement.} The main part of this work was done during a stay of
          both authors at the Mathematische Forchungsinstitut in
          Oberwolfach in January 2010 in the framework of the `Research in Pairs' programme. We want to thank the MFO for the perfect
          working conditions.

\section{Pairing extensions with asymptotic homomorphisms}

\subsection{Asymptotic homomorphisms}\label{asymp} Let $A$ and $B$ be
$C^*$-algebras, $A$ separable. As in \cite{CH} we define an \emph{asymptotic
  homomorphism} $\alpha : A \to B$ to be a path of maps $\alpha_t : A
\to B$, $t \in [1,\infty)$, such that
\begin{itemize}
\item $ t \mapsto \alpha_t(a)$ is continuous,
\item $\lim_{t \to \infty} \alpha_t(a + \lambda b) - \alpha_t(a)
  - \lambda \alpha_t(b) = 0$,
\item$\lim_{t \to \infty} \alpha_t(ab) - \alpha_t(a)\alpha_t(b)
  = 0$, and
\item $\lim_{t \to \infty} \alpha_t(a^*) - \alpha_t(a)^* = 0$
\end{itemize}
for all $a,b \in A$ and all $\lambda \in \mathbb C$. It follows from
these conditions that $\limsup_t \left\|\alpha_t(a)\right\| \leq
\left\|a\right\|$, and hence in particular that $\sup_{t \in
  [1,\infty)} \left\| \alpha_t(a)\right\| < \infty$ for all $a \in
A$.

We say that an asymptotic homomorphism $\alpha : A \to B$ is
\emph{equi-continuous} when $\alpha_t, t \in [1,\infty)$, is an
equi-continuous family of maps. By a standard argument any asymptotic
homomorphism $\alpha$ is asymptotic to an equi-continuous asymptotic
homomorphism $\alpha'$, i.e. $\alpha'$ is equi-continuous and $\lim_{t
  \to \infty} \alpha_t(a) - \alpha'_t(a) = 0$ for all $a \in A$. Hence
we may assume, as we shall, that all asymptotic homomorphisms under
consideration are equi-continuous. It
will also be convenient for us, if only as a tool, to deal with asymptotic homomorphisms $\alpha$
which are both equi-continuous and \emph{uniformly continuous} in the
sense that $t  \mapsto \alpha_t(a)$ is uniformly continuous for all $a
\in A$. We shall need the following lemma in order to fully exploit
this additional property.

\begin{lemma}\label{uniformlift} Let $D$ be a $C^*$-algebra containing
  a $\sigma$-unital ideal $D_0$, and let $q: D \to D/D_0$ be the
  quotient map. Let $\varphi =
  \left(\varphi_t\right)_{t \in [1,\infty)} : A \to D/D_0$ be a
  uniformly continuous asymptotic homomorphism. There
  is then a family $\overline{\varphi}_t : A \to D, t \in
  [1,\infty)$, of maps such that
\begin{enumerate}
\item[i)] $q \circ \overline{\varphi}_t = \varphi_t$ for all $t \in
  [1,\infty)$,
\item[ii)] $\overline{\varphi}_t : A \to D$, $t \in [1,\infty)$, is
  equi-continuous,
\item[iii)] $t \mapsto \overline{\varphi}_t(a)$ is uniformly continuous
  for all $a \in A$, and
\item[iv)] $\sup_{t \in [1,\infty)}
  \left\|\overline{\varphi_t}(a)\right\| < \infty$ for all $a \in A$.
\end{enumerate}

\begin{proof} Let $\psi = \left(\psi_t\right)_{t \in
    [1,\infty)} : A \to D$ be an equi-continuous lift of
  $\varphi$ such that $\sup_{t \in [1,\infty)}
  \left\|\psi_t(a)\right\| < \infty$ for all $a \in A$. $\psi$ exists by Lemma 2.1 of \cite{MT2}. Let $F_1
  \subseteq F_2 \subseteq F_3 \subseteq \dots$ be a sequence of finite
  sets with dense union in $A$. Let $u_1 \leq u_2 \leq u_3 \leq \dots$ be an approximate
  unit in $D_0$ such that
$$
\left\|(1-u_n)\left(\psi_t(a) - \psi_{t'}(a)\right)\right\| \leq
\left\|\varphi_t(a) - \varphi_{t'}(a)\right\| + \frac{1}{n}
$$
for all $a \in F_n$ and all $t,t' \in [1,n+1]$. Such an approximate unit exists
because $D_0$ is $\sigma$-unital. For $t \in [n,n+1]$, set
$v_t = (t-n)u_{n+1} + (n+1-t)u_n$, and define $\overline{\varphi}_t :
A \to M(D)$ such that
$$
\overline{\varphi}_t(a) = (1-v_t)\psi_t(a) .
$$
It is obvious that $\left(\overline{\varphi}_t\right)_{t \in
  [1,\infty)}$ is equi-continuous since $\left(\psi_t\right)_{t \in
  [1,\infty)}$ is and that i) and iv) hold. To check that $\overline{\varphi}$ is uniformly
continuous, let $a \in A$ and $\epsilon > 0$ be given. By
equi-continuity there is a $b \in F_k$ such that $\frac{1}{k} \leq
\epsilon$ and $\left\|\overline{\varphi}_t(a) -
  \overline{\varphi}_t(b)\right\| \leq \epsilon$ and $\left\|{\psi}_t(a) -
  {\psi}_t(b)\right\| \leq \epsilon$ for all $t \in
[1,\infty)$. Let $t \geq k$. If $|t' -t| \leq 1$ we find that
\begin{equation*}\label{est1}
\begin{split}
&\left\|\overline{\varphi}_t(a) - \overline{\varphi}_{t'}(a)\right\|
\leq \left\|\overline{\varphi}_t(b) - \overline{\varphi}_{t'}(b)\right\|
+ 2 \epsilon \\
&\leq \left\|(1-v_t)\left({\psi}_t(b) - \psi_{t'}(b)\right)\right\| +
\left\|\psi_{t'}(b)(v_{t'} - v_t)\right\|  + 2 \epsilon\\
& \leq \left\| \varphi_t(b) - \varphi_{t'}(b)\right\| + \frac{1}{k} +
|t-t'| \sup_{s \in [1,\infty)} \left\|\psi_s(b)\right\| + 2\epsilon \\
& \leq \left\| \varphi_t(a) - \varphi_{t'}(a)\right\| + 5\epsilon +
|t-t'| \sup_{s \in [1,\infty)} \left\|\psi_s(a)\right\| + |t-t'|
\epsilon .
\end{split}
\end{equation*}
By uniform continuity of $t \mapsto \varphi_t(a)$ this shows there is a $\delta >
0$ such that
$$
\left\|\overline{\varphi}_t(a) -
  \overline{\varphi}_{t'}(a)\right\| \leq 6\epsilon + \epsilon \sup_{s
  \in [1,\infty)} \left\|\psi_s(a)\right\|
$$
when $t \geq k$ and
$\left|t-t'\right| \leq \delta$. Since $[1,k]$ is compact, we see that
$t \mapsto \overline{\varphi}_t(a)$ is uniformly continuous on
$[1,\infty)$.
\end{proof}
\end{lemma}

\subsection{Folding}\label{folding} Let $A$ and $B$ be $C^*$-algebras, $A$ separable,
$B$ $\sigma$-unital. Let $M(B)$ be
the multiplier algebra of $B$ and $q_B : M(B) \to Q(B)$ the quotient
map onto the generalised Calkin algebra $Q(B) = M(B)/B$. As in
\cite{MT2} an asymptotic homomorphism $\varphi : A \to Q(B)$ will be
called an \emph{asymptotic extension}. In \cite{MT2} we used a
construction called \emph{folding} which produces a genuine extension
out of an asymptotic one. To introduce it here, let $\varphi : A \to Q(B)$ be a equi-continuous asymptotic extension. A \emph{lift}
of $\varphi$ is an equi-continuous family of maps $\overline{\varphi}_t : A \to M(B)$,
$t \in [1,\infty)$, such that $\sup_t
\left\|\overline{\varphi}_t(a)\right\| < \infty$ for all $a \in A$ and
$q_B \circ \overline{\varphi}_t =
\varphi_t$ for all $t$. The existence of such a lift follows from Lemma 2.1 of \cite{MT2}.


Let $b$ be a strictly positive
element in $B$, $0 \leq b \leq 1$, which exists because we assume that
$B$ is $\sigma$-unital. As in \cite{MT2} a \emph{unit sequence} is a sequence
$u_0 \leq u_1 \leq u_2 \leq \dots$
of elements in $B$ such that
\begin{itemize}
\item  $u_n = f_n(b)$ for some $f_n \in C[0,1], 0 \leq f_n \leq 1$,
  which is zero in a neighbourhood of $0$,
\item $u_{n+1}u_n = u_n$ for all $n$, and
\item  $\lim_{n \to \infty} u_n b = b$.
\end{itemize}
The existence of a unit sequence, with some important additional properties that we
shall need is a consequence of the following well-known lemma.

\begin{lemma}\label{deltaepsilon} Let $K \subseteq B$ and $L \subseteq
M(B)$ be compact in the norm topology, and let $\delta > 0$ and
$\epsilon > 0$ be arbitrary. It follows that there is a continuous
function $f : [0,1] \to [0,1]$ such that
\begin{enumerate}
\item[i)] $f$ is zero in an open neighbourhood of $0$,
\item[ii)] $f(t) = 1, \ t \geq \delta$,
\end{enumerate}
and $u = f(b) \in B$ has the property that
\begin{enumerate}
\item[iii)] $\left\|um-mu\right\| \leq \epsilon \ \forall m \in L$ and
\item[iv)] $\left\|uk-k\right\| \leq \epsilon \ \forall k \in K$.
\end{enumerate}
\begin{proof} See for example Lemma 7.3.1 in \cite{BO}.
\end{proof}
\end{lemma}

Given a unit sequence $\{u_n\}$ we set $\Delta_0 = \sqrt{u_0}$ and
$\Delta_n = \sqrt{u_n - u_{n-1}}, n \geq 1$. Then
\begin{enumerate}
\item[a1)]\label{enu8} $\Delta_i \Delta_j = 0$ when $|i -j| \geq 2$, and
\item[a2)]\label{enu9} $\sum_{j=0}^{\infty} \Delta_j^2 = 1$, with convergence in
  the strict topology.
\end{enumerate}
In particular, it follows that
\begin{enumerate}[resume]
\item[a3)]\label{enu10}  $\sum_{j=0}^{\infty} \Delta_i \Delta_j^2 =
  \sum_{l=i-1}^{i+1} \Delta_i\Delta_l^2 = \Delta_i$
\end{enumerate}
for all $i$, including $i = 0$ when we set $\Delta_{-1} = 0$.

A \emph{discretization} (of $[1,\infty)$) is an increasing sequence $t_0
\leq t_1 < t_2 < \dots $ in
 $[1,\infty)$ such that
\begin{enumerate}
\item[a4)]\label{enu14} $\lim_{n \to \infty} t_n = \infty$,
\item[a5)]\label{enu15} $\lim_{n \to \infty} t_{n+1} - t_n = 0$,
and
\item[a6)]\label{enu16} $t_n \leq n$ for all $n \geq 1$.
\end{enumerate}

When $\varphi : A \to Q(B)$ is an asymptotic extension and
$\overline{\varphi} : A \to M(B)$ is a lift of $\varphi$ a pair
$\left(\{u_n\}, \{t_n\}\right)$, where $\{u_n\}$ is a unit
sequence and $\{t_n\}$ a discretization, is said to be \emph{compatible} with
$\overline{\varphi}$ when
\begin{equation}\label{compatible1}
\lim_{n \to \infty} \sup_{t \in [1,n+2]} \left\|u_n\overline{\varphi}_t(a) -
  \overline{\varphi}_t(a)u_n\right\| ,\end{equation}
and
\begin{equation}\label{uniccomp}
\lim_{t \to \infty} \sup_{t \in \left[t_n,t_{n+1}\right]}
\left\|\overline{\varphi}_t(a) - \overline{\varphi}_{t_n}(a)\right\| = 0
\end{equation}
for all $a \in A$, and
\begin{equation}\label{compatible2}
\lim_{n \to \infty} \sup_{t \in [1,n+2]}
\left[\left\|(1-u_n)f(t)\right\| - \left\|q_B\left(f(t)\right)
  \right\|\right] = 0
\end{equation}
for all $f \in C_b\left(\left[1,\infty\right),M(B)\right)$ of the form
\begin{itemize}
\item $f(t) =\overline{\varphi}_t(a)\overline{\varphi}_t(b) -
      \overline{\varphi}_t(ab)$,
\item $f(t) = \overline{\varphi}_t(a) + \lambda\overline{\varphi}_t(b) -
      \overline{\varphi}_t(a +\lambda b)$, and
\item $f(t) = \overline{\varphi}_t(a^*) -
    \overline{\varphi}_t(a)^*$
 \end{itemize}
for any elements $a,b \in A$, $\lambda\in\mathbb C$. The existence
of compatible pairs $\left(\{u_n\}, \{t_n\}\right)$ was
established in \cite{MT1} and \cite{MT2}. Note that condition
(\ref{uniccomp}) is automatically fulfilled when
$\overline{\varphi}$ is uniformly continuous; it follows then from a5).

Assume that $\left(\{u_n\}, \{t_n\}\right)$ is a pair compatible with
$\overline{\varphi}$. The combined triple
$ f = \left(\overline{\varphi}, \{u_n\}, \{t_n\}\right)$ will be
called \emph{folding data} for the asymptotic extension
$\varphi$. We can then define $\overline{\varphi}_f: A \to M(B)$ such
that
$$
\overline{\varphi}_f(a) = \sum_{j=0}^{\infty} \Delta_j
\overline{\varphi}_{t_j}(a)\Delta_j ,
$$
cf. Lemma 3.1 of \cite{MT2}. By Lemmma 3.5 in \cite{MT2},
$$
\varphi_f = q_B \circ  \overline{\varphi}_f
$$
is an extension $\varphi_f: A \to Q(B)$ which we call a
\emph{folding} of $\varphi$.

A \emph{re-parametrisation} is a non-decreasing continuous function
$r : [1,\infty) \to [1,\infty)$ such that $\lim_{t \to \infty} r(t) =
\infty$. If there is a constant $K$ such that $\left|r(s) -
  r(t)\right| \leq K|s-t|$ for all $s,t\in [1,\infty)$ we say that $r$
is \emph{Lipschitz}. Note that when $\varphi : A \to Q(B)$ is an
asymptotic extension and $\overline{\varphi} : A \to M(B)$ is a
lift of $\varphi$, we can define a new asymptotic extension
$\varphi^r$ with a lift $\overline{\varphi^r}$ such that
$\varphi^r_t = \varphi_{r(t)}$ and $\overline{\varphi^r}_t =
\overline{\varphi}_{r(t)}$. Both $\varphi^r$ and
$\overline{\varphi^r}$ remain uniformly continuous if
$\overline{\varphi}$ is uniformly continuous and $r$ is
Lipschitz.

\begin{lemma}\label{uniform7} Let $\varphi : A \to Q(B)$ be an
  asymptotic extension and let $f= \left(\overline{\varphi},
    \{u_n\}, \{t_n\}\right)$ be folding data for
    $\varphi$. There is then a Lipschitz re-parametrisation
    $r$  and a
    discretization $\left\{s_n\right\}$ such that
\begin{enumerate}
\item[i)] $\varphi^r$ and $\overline{\varphi^r}$ are both uniformly
  continuous, and
\item[ii)] $r(s_n) = t_n$ for all $n$.
\end{enumerate}
Furthermore, $f' = \left(\overline{\varphi^r}, \left\{u_n\right\},
  \left\{s_n\right\}\right)$ is folding data for $\varphi^r$ and
$\varphi^r_{f'} = \varphi_f$.
\begin{proof} Let $F_1 \subseteq F_2  \subseteq F_3 \subseteq \dots $
  be an increasing sequence of finite sets with dense union in
  $A$. It follows from (\ref{uniccomp}) that there is an increasing
  sequence $1 < m_1 < m_2 < m_3 < \dots$ in $\mathbb N$ such that
\begin{equation}\label{uui}
\sup_{t \in \left[t_n ,t_{n+k}\right]} \left\| \overline{\varphi}_t(a) - \overline{\varphi}_{t_n}(a) \right\|
\leq \frac{1}{k}
\end{equation}
for all $a \in F_k$ and all $n \geq m_k$. By increasing the $m_k$'s we
can arrange that $m_{k+1} - m_k = kl_k$ for some $l_k \in \mathbb
N$. Set $l_0 = m_1$. Thanks to a5) we can
also arrange that
\begin{equation}\label{tlip}
t_{n+k} - t_n \leq \frac{1}{k}
\end{equation}
for all $n \geq m_k$. Set
\begin{equation}\label{sformel}
s_{{m_k} + i} = \sum_{j=0}^{k-1} l_j + \frac{i}{k}
\end{equation}
for all $i \in
\left\{0,1,2,\dots, kl_k\right\}$ and all $k = 1,2,3, \dots$. Then
$s_n \leq n$ for all $n \geq m_1$. For $j
\in \{0,1,\dots, m_1-1\}$ we choose
$s_j \in \left[1,m_1\right]$
such
that $s_j$ increases with $j$ and $s_j \leq j$ for all $j\in \left\{
  1,2, \dots ,m_1-1\right\}$. Then
$\{s_n\}$ is a discretization. Define
$r : [1,\infty) \to [1,\infty)$ such that ii) holds and $r$ is linear
on $\left[s_n,s_{n+1}\right]$ for all $n$. Then $r$ is a re-parametrisation and
$$
r\left(\sum_{j=0}^{k-1} l_j + i\right) = t_{m_{k} + ik},
$$
for all $i \in \{0,1, \dots, l_k\}$ and all $k$. It follows then from (\ref{uui}), by use of the
equi-continuity of $\overline{\varphi}$, that
\begin{equation}\label{uuii}
\lim_{t \to \infty} \sup_{ v \in [0,1]} \left\| \overline{\varphi}_{r(t+v)}(a) -
  \overline{\varphi}_{r(t)}(a)\right\| = 0
\end{equation}
for all $a \in A$. (\ref{uuii}) implies that
$\overline{\varphi^r}$, and hence also $\varphi^r$ are uniformly
continuous, i.e. i) also holds. It follows from
  (\ref{tlip}) and (\ref{sformel}) that
$$
\frac{r\left(s_{m_k + i+1}\right) - r\left(s_{m_k +i}\right)}{s_{m_k+i+1} - s_{m_k+i}} \leq 1
$$
when $i \in \left\{0,1,2,\dots, kl_k -1\right\}$. It follows that
there is a $K > 0$ such that $r\left(s_{j+1}\right) -
r\left(s_j\right) \leq K\left(s_{j+1} - s_j\right)$
for all $j \geq 0$, proving
that $r$ is Lipschitz. Finally, it is now straightforward to check
that $f' = \left(\overline{\varphi^r}, \left\{u_n\right\},
  \left\{s_n\right\}\right)$ is folding data for $\varphi^r$ and that
$\varphi^r_{f'} = \varphi_f$.
\end{proof}
\end{lemma}

There is an alternative picture of the folding operation which we
shall need. Let $l^2(B)$ denote the standard Hilbert $B$-module of
'square-summable' sequences $\left(b_0,b_1,b_2,\dots \right)$ of
elements from $B$. The $C^*$-algebra of adjoint-able operators on
$l^2(B)$ can be identified with $M(B \otimes \mathbb K)$ and the 'compact'
operators on $l^2(B)$ is then identified with $B\otimes \mathbb
K$, cf. \cite{K2}. By using the standard matrix units
$\left\{e_{ij}\right\}_{i,j=0}^{\infty}$ which act on $l^2(B)$ in
the obvious way, we can use a set of folding data $ f =
\left(\overline{\varphi}, \{u_n\}, \{t_n\}\right)$ to define a map
$\overline{\varphi}^f : A \to M\left(B \otimes \mathbb K\right)$
such that
$$
\overline{\varphi}^f(a) = \sum_{i=0}^{\infty} \sum_{j= i-1}^{i+1}
\Delta_i \overline{\varphi}_{t_i}(a) \Delta_j \otimes e_{ij} .
$$
The sum converges in the strict topology because $\sup_{i,j}
\left\|\Delta_i \overline{\varphi}_{t_i}(a)\Delta_j\right\| <
\infty$. $\overline{\varphi}^f$ is continuous by equi-continuity of
$\overline{\varphi}_t, t \in [1, \infty)$, and a direct check, as in
the proof of Lemma 3.5 of \cite{MT2}, shows that
$\overline{\varphi}^f$ is a $*$-homomorphism modulo $B \otimes \mathbb
K$, i.e. $\varphi^f = q_{B \otimes \mathbb K} \circ
\overline{\varphi}^f$ is an extension of $A$ by $B \otimes \mathbb
K$. In order to see the relation between $\varphi^f$ and
$\varphi_f$, observe that
$$
V = \sum_{j=0}^{\infty} \Delta_j \otimes e_{j0}
$$
 is a partial isometry in $M(B \otimes \mathbb K)$ such that
\begin{itemize}
\item
$V \left(\overline{\varphi}_f(a) \otimes e_{00}\right) V^* - \overline{\varphi}^f(a)
\in B \otimes \mathbb K$, and
\item $V^*V = 1 \otimes e_{00}$.
\end{itemize}
Since $(1 -VV^*)(l^2(B)) \oplus l^2(B)$ and $\left(1 - 1 \otimes
  e_{00}\right)(l^2(B)) \oplus l^2(B)$ are isomorphic Hilbert
$B$-modules by Kasparov's stabilisation theorem, cf. \cite{K2}, it
follows that there is a
unitary dilation $U$ of $V$, acting on $l^2(B) \oplus l^2(B)$,
such that
$$
U \left(  \begin{matrix} \overline{\varphi}_f(a) \otimes e_{00} & 0 \\
    0 & 0 \end{matrix} \right)U^* - \left(  \begin{matrix} \overline{\varphi}^f(a) & 0 \\
    0 & 0 \end{matrix} \right) \in M_2\left(B \otimes \mathbb K\right)
.
$$

In this way we obtain the following conclusion.

\begin{lemma}\label{alt} Assume that $B$ is stable, and identify $B
  \otimes \mathbb K$ with $B$.

Then $\varphi_f \oplus 0$ is unitarily equivalent to $\varphi^f \oplus
0$.
\end{lemma}


\subsection{A key lemma}

Two asymptotic extensions $\varphi, \varphi' : A \to Q(B)$ are
\emph{strongly homotopic} when they define the same element in
$[[A,Q(B)]]$. This means that there is an asymptotic homomorphism
$\alpha : A \to C[0,1] \otimes Q(B)$ such that $\ev_0 \circ \alpha_t
= \varphi_t$ and $\ev_1 \circ \alpha_t = \varphi'_t$ for all $t$, where $\ev_s :
C[0,1] \otimes Q(B) \to Q(B)$ is evaluation at $s\in [0,1]$. In this
subsection we will relate a particular folding of $\varphi$ to a
folding of $\varphi'$, assuming that the strong homotopy $\alpha$
connecting $\varphi$ to $\varphi'$ is uniformly continuous. By
Lemma \ref{uniformlift}, applied with $D = C[0,1] \otimes M(B)$
and $D_0 = C[0,1]\otimes B$, there is an equi-continuous and
uniformly continuous lift $\overline{\alpha} : A \to C[0,1]\otimes
M(B)$ of $\alpha$ such that $\sup_{t \in [1,\infty)}
\left\|\overline{\alpha}_t(a)\right\| < \infty$ for all $a \in A$. Let
$\{u_n\}$ be a unit sequence in $B$ such that
\begin{equation}\label{functions}
\lim_{n \to \infty} \sup_{t \in [1,n+2]} \left[
\sup_{s\in [0,1]}\left\|\left(1-u_n\right)f(t)(s)\right\| -
\left\|q_{C[0,1] \otimes B}\left(f(t)\right)
\right\|\right] = 0
\end{equation}
when $f \in C_b\left([1,\infty), C[0,1]\otimes M(B)\right)$ is any of the following functions:
\begin{itemize}
\item $f(t) =\overline{\alpha}_t(a)\overline{\alpha}_t(b) -
      \overline{\alpha}_t(ab)$,
\item $f(t) =\overline{\alpha}_t(a) + \lambda\overline{\alpha}_t(b) -
      \overline{\alpha}_t(a+\lambda b)$, or
\item $ f(t) = \overline{\alpha}_t(a^*) -
  \overline{\alpha}_t(a)^*$
\end{itemize}
for any $a,b \in A$, $\lambda\in\mathbb C$. Furthermore, we
require also that
\begin{equation} \label{enu23} \lim_{n \to \infty} \sup \left\{\left\|
      u_n\overline{\alpha}_t(a)(s) - \overline{\alpha}_t(a)(s)u_n
    \right\|: \ t \in [1,n+2], \ s \in [0,1] \right\}= 0
\end{equation}
for all $a \in A$. That such a unit sequence exists follows from the
separability of $A$ and the equi-continuity of $\overline{\alpha}$ by
use of Lemma \ref{deltaepsilon}.

Let $\left\{t'_j\right\}$ and $\left\{t_j\right\}$ be discretizations. Then $f_0 = \left(\ev_0 \circ \overline{\alpha}, \left\{u_n\right\},
  \left\{t'_n\right\} \right)$ is folding data for $\ev_0 \circ \alpha$ and $f_1 = \left(\ev_1 \circ \overline{\alpha}, \left\{u_n\right\},
  \left\{t_n\right\} \right)$ is folding data for $\ev_1 \circ
\alpha$. The key lemma referred to in the title of this section is

\begin{lemma}\label{vladlemma} In the above setting, assume that $\kappa : A \to Q(B)$ is an
  extension such that $\left(\ev_0 \circ \alpha\right)_{f_0} \oplus
  \kappa$ is asymptotically split. It follows that $\left(\ev_1 \circ \alpha\right)_{f_1} \oplus
  \kappa$ is asymptotically split.
\end{lemma}

For the proof we need the following

\begin{lemma}\label{disclemma} Let $\left\{t'_n\right\}$ and
  $\left\{t_n\right\}$ be two discretizations and $a_1
  < a_2 < \dots $ a strictly increasing sequence in $\mathbb N$.

There is a sequence $h_0 \leq h_1 \leq h_2 \leq \dots$ of
continuous functions $h_j : [1,\infty) \to [1,\infty)$ such that
\begin{enumerate}
\item[i)] $h_j(t) = t'_j, \ j \leq a_k, \ t \in [k,k+1]$,
\item[ii)] $h_{j+1}(t) - h_j(t) \leq \max \left\{\frac{1}{k}, t'_{j+1}
  - t'_j\right\} \ \forall j , \
  t \in [k,k+1]$,
\item[ii)] for all $n \in \mathbb N$ there is an $N_n \in \mathbb N$
  such that $h_j(t) =t_j$ when $t \in [1,n]$, $j \geq N_n$, and
\item[iv)] $h_j(t) \leq j$ for all $j \geq 1$ and all $t$.
\end{enumerate}
\begin{proof} Let $k \in \mathbb N$. Since $\lim_{j \to \infty}
  t'_{j+1} -t'_j = 0$ and $\lim_{j \to \infty} t_{j+1} - t_j = 0$
  there is a $b_k \geq a_k$ such that $\max \left\{ t'_{j+1} -
  t'_j, t_{j+1} - t_j\right\} \leq \frac{1}{k}$ for all $j \geq
b_k$. We arrange that $b_{k+1} > b_k$. On the interval $[k,k+1]$ we set $h_j(t) = t'_j$ when $j \leq
  b_k$. Since $\lim_{j \to \infty} t_{j+1} - t_j = 0$ and $\lim_{j \to
    \infty} t_j = \infty$
    there is an $m_k
  > b_k$
  such that $\frac{n}{k} + t'_{b_k} \geq t_n > t'_{b_k}$ for all $n \geq m_k$. We set
$$
h_j(t) = \max \left\{ \min \{ \frac{j}{k} + t'_{b_k}, \ t_j\}, \
  t'_{b_k} \right\}
$$
when $\ j > b_k$ and $t \in \left[k,k+\frac{1}{2}\right]$. Then
$h_j(t) = t_j$ when $j \geq m_k$. With these
choices we have defined the $h_j$'s on all the intervals
$\left[k,k+\frac{1}{2}\right], k = 1,2,3, \dots$, but it remains to
define the $h_j$'s on $\left[k+\frac{1}{2},k+1\right]$ when $j >
b_k$. For this note that $h_{j+1}\left(k+\frac{1}{2}\right) -
h_j\left(k + \frac{1}{2}\right) \leq \frac{1}{k}$ and
$$
h_{j+1}\left(k+1\right) -
h_j\left(k + 1\right) \leq \max \left\{\frac{1}{k +1}, t'_{j+1} -t'_j \right\} \leq \frac{1}{k}
$$
for all $j \geq b_k$. Hence by defining $h_j, \ j > b_k$, to be the linear function
on $\left[k+\frac{1}{2},k+1\right]$ which connects
$h_j\left(k+\frac{1}{2}\right)$ to $h_j\left(k+1\right)$ we have
obtained what we wanted.
\end{proof}
\end{lemma}

\emph{ Proof of Lemma \ref{vladlemma}.} Set $\alpha_t^s = \ev_s \circ \overline{\alpha}_t$,
  $\Delta_0 = \sqrt{u_0}$ and $\Delta_n = \sqrt{u_n - u_{n-1}}, n \geq
  1$. We define $\psi^0 : A \to M(B)$ such that
$$
\psi^0(a) =  \sum_{j=0}^{\infty} \Delta_j \alpha^0_{t'_j}(a)\Delta_j
.
$$
Note that $\psi^0$ is continuous thanks to the equi-continuity of
$\ev_0 \circ \overline{\alpha}_t, t\in [1,\infty)$, cf. Lemma 3.1 of \cite{MT2}, and that
$q_B \circ \psi^0 = \left(\ev_0 \circ \alpha\right)_{f_0}$. It
follows from our assumption that there is an equi-continuous asymptotic homomorphism
$\mu :  A \to M_2\left(M(B)\right)$ such that
$$
\mu_t = \left( \begin{matrix} \mu^{11}_t & \mu^{12}_t \\ \mu^{21}_t &
  \mu^{22}_t \end{matrix} \right),
$$
where $q_B \circ \mu^{11}_t = q_B \circ \psi^0, q_B \circ \mu^{12}_t =
  q_B \circ \mu^{21}_t = 0$ and $q_B \circ \mu^{22}_t = \kappa$ for
  all $t \in [1, \infty)$. It follows from Lemma \ref{deltaepsilon}
  that we can choose a continuous path $v_t, t\in [1, \infty)$, in the
 $C^*$-subalgebra of $B$ generated by the strictly positive
 element $b$ such that $t \mapsto v_t$ is
 norm-continuous, $0\leq v_t\leq 1$ for all $t$, and
\begin{enumerate}
\item[a7)] \label{enu24} $v_t \Delta_i = \Delta_i, \ i \leq t$,
\item[a8)] \label{enu25} $\lim_{t \to \infty} \left\|v_t
      \mu^{11}_t(a) -\mu^{11}_t(a)v_t\right\| = 0$ for all $a \in A$,
\item[a9)] \label{enu126} $\lim_{t \to \infty} \left\|\left(1 - v_t\right)
    \mu^{12}_t(a)\right\| = \lim_{t \to \infty} \left\|\left(1 - v_t\right)
    \mu^{21}_t(a)\right\| = 0$ for all $a \in A$,
\item[a10)] \label{enu26} $ \lim_{t \to \infty}  \left\|v_t\psi^0(a) - \psi^0(a)v_t \right\| = 0$ for all $a \in A$, and
\item[a11)] \label{enu27} $\lim_{t \to \infty} \left\|(1-v_t)\left[\mu^{11}_t(a) -
      \psi^0(a)\right] \right\| = 0$ for all $a \in A$.
\end{enumerate}
Since $\lim_{i \to \infty} v_t \Delta_i = 0$ for all $t$ there is an
increasing function $i : \mathbb N \to \mathbb N$ such that
\begin{equation}\label{vdelta}
\lim_{ k  \to \infty} \sup_{j \geq i(k)} \sup_{t \in [k,k+1]}\left\|v_t \Delta_j\right\| =
0 .
\end{equation}
Let $F_1 \subseteq F_2 \subseteq F_3 \subseteq \dots$ be a sequence of
finite subsets with dense union in $A$. For each $n \in \mathbb N$
there is an $\epsilon_n > 0$ such that
\begin{equation}\label{alpha}
\left\| \alpha^s_t(a) - \alpha^{s'}_t(a)\right\| \leq \frac{1}{n}
\end{equation}
when $\left|s-s'\right| \leq \epsilon_n$, $t \in [1,n+2], a \in
F_n$. Choose then a sequence of continuous non-increasing functions $g_k : [1,\infty)
\to [0,1], \ k = 0,1,2,3,\dots$, such that
\begin{enumerate}
\item[a12)] \label{enu28} for each $t \in [1,\infty)$, $g_k(t) = 1$ for all but
  finitely many $k$,
\item[a13)] \label{enu29} $g_i(t) = 0$ for all $i = 1,2,\dots, i(k)$, when $t \geq k$,
\item[a14)] \label{enu30} $g_k \leq g_{k+1}$ for all $k$, and
\item[a15)] \label{enu31} $g_{k+1}(t) - g_k(t) \leq \epsilon_n$ when $t \in [1,n+2]$,
  for all $k,n$.
\end{enumerate}

Since the $g_k$'s are non-increasing it follows from a12)
that there are numbers $a_1 <a_2 < a_3 < \dots $ in $\mathbb N$
such that
\begin{equation}\label{ak0}
a_k \geq i(k)
\end{equation}
and
\begin{equation}\label{ak}
g_j(t) = 1 , \ t \in [k,k+1], \ j \geq a_k -1 .
\end{equation}
We can then use Lemma \ref{disclemma} to obtain continuous functions
$h_j : [1,\infty) \to [1,\infty), j = 0,1,2,3, \dots $, such that $h_0 \leq h_1 \leq h_2 \leq \dots$ and
\begin{enumerate}
\item[a16)] \label{enu32} $h_j(t) = t'_j, \ j \leq a_k, \ t \in [k,k+1]$,
\item[a17)] \label{enu33} $h_{j+1}(t) - h_j(t) \leq \max \left\{\frac{1}{k}, t'_{j+1}
  - t'_j\right\} \forall j , \
  t \in [k,k+1]$,
\item[a18)] \label{enu34} for all $n \in \mathbb N$ there is an $N_n \in \mathbb N$
  such that $h_j(t) =t_j$ when $t \in [1,n]$ and $j \geq N_n$, and
\item[a19)] \label{enu134} $h_j(t) \leq j$ for all $j,t$.
\end{enumerate}
Now we set
$$
\psi_t(a) = \sum_{j=0}^{\infty}  \Delta_j
\alpha^{g_j(t)}_{h_j(t)}(a)\Delta_j .
$$
It follows from Lemma 3.1 of \cite{MT2} and the equi-continuity of
$\overline{\alpha}$ that $\psi_t : A \to M(B), \ t\in [1,\infty)$, is
an equi-continuous family. Furthermore, it follows from a12)
and a18) that for each $n \in \mathbb N$ there
is an $N'_n \in \mathbb N$ such that
\begin{equation}\label{part}
\psi_t(a) = \sum_{j=0}^{N'_n}  \Delta_j
\alpha^{g_j(t)}_{h_j(t)}(a)\Delta_j + \sum_{j = N'_n+1}^{\infty}
\Delta_j \alpha^1_{t_j}(a)\Delta_j
\end{equation}
for all $t \leq n$ and all $a \in A$. In particular, this shows that $q_B \circ \psi_t
= \left(\ev_1 \circ \alpha\right)_{f_1}$ and that $t \mapsto \psi_t(a)$ is
norm-continuous.

Let
$$
\lambda^{11}_t(a) = v_t\mu^{11}_t(a)v_t +
\left(1-v_t^2\right)^{\frac{1}{2}} \psi_t(a)\left(1-v_t^2\right)^{\frac{1}{2}}
$$
and set $\lambda^{ij}_t(a) = \mu^{ij}_t(a)$ for $(i,j) \neq
(1,1)$. Finally, set
$$
\Lambda_t(a) = \left( \begin{matrix} \lambda^{11}_t(a) &
    \lambda^{12}_t(a) \\ \lambda^{21}_t(a) & \lambda^{22}_t(a)
  \end{matrix} \right) .
$$
Then $\Lambda_t : A \to M_2\left(M(B)\right), t \in [1,\infty)$, is an
equi-continuous family of maps and since $q_B \circ \lambda^{11}_t =
\left(\ev_1 \circ \alpha\right)_{f_1}$, $q_B \circ \lambda^{12}_t =
q_B \circ \lambda^{21}_t = 0$ and $q_B \circ \lambda^{22}_t = \kappa$,
it suffices to show that $\left(\Lambda_t\right)_{t \in [1,\infty)}$
is an asymptotic homomorphism.

Each $\mu^{ij}$ is asymptotically
linear and $\mu^{ij}_t(a^*)$ agrees asymptotically with
$\mu^{ji}_t(a)^*$ since $\mu$ is an asymptotic
homomorphism. Furthermore, it follows from a7) and (\ref{functions}) that $a \mapsto
\left(1-v_t^2\right)^{\frac{1}{2}}
\psi_t(a)\left(1-v_t^2\right)^{\frac{1}{2}}$ is asymptotically linear
and asymptotically commutes with the involution as $t$ tends to
infinity. It is then clear that the same is true for $\lambda^{11}$
and hence for $\Lambda$.

We check that $\Lambda$ is
asymptotically multiplicative. For this we write $A(t) \sim B(t)$
between $t$-dependent elements from $M(B)$ when $\lim_{t \to \infty}
\left\|A(t) -B(t)\right\| = 0$. Let $a,b \in A$. It suffices to show that
\begin{equation}\label{tjek1}
\lambda^{11}_t(a)\lambda^{12}_t(b) +
\lambda^{12}_t(a)\lambda^{22}_t(b) \sim \lambda^{12}_t(ab)
\end{equation}
and
\begin{equation}\label{tjek2}
\lambda^{11}_t(a)\lambda^{11}_t(b) +
\lambda^{12}_t(a)\lambda^{21}_t(b) \sim \lambda^{11}_t(ab) .
\end{equation}

To handle (\ref{tjek1}) observe that $v_t
\mu^{11}_t(a)v_t\mu^{12}_t(b) \sim \mu^{11}_t(a)\mu^{12}_t(b)$ by a8) and a9), and that $\left(1-v_t^2\right)^{\frac{1}{2}}
\psi_t(a)\left(1-v_t^2\right)^{\frac{1}{2}} \mu^{12}_t(b) \sim 0$ by a9). Since $\mu$ is an asymptotic homomorphism we have also that
$\mu^{11}_t(a)\mu^{12}_t(b) + \mu^{12}_t(a)\mu^{22}_t(b) \sim
\mu^{12}_t(ab) = \lambda^{12}_t(ab)$. (\ref{tjek1}) follows from this.

It remains to verify (\ref{tjek2}). For this we prove first that
\begin{equation}\label{A1}
v_t\left(\psi_t(c) - \psi^0(c)\right) \sim  \left(\psi_t(c) -
  \psi^0(c)\right)v_t \sim  0
\end{equation}
for all $c \in A$. To establish (\ref{A1}) observe that the following estimate is valid when $t \geq k$:
\begin{equation*}
\begin{split}
&\left\|v_t\left(\psi_t(c) - \psi^0(c)\right)\right\|^2 =
\left\| v_t \left( \sum_{j > i(k)} \Delta_j \left(
      \alpha^{g_j(t)}_{h_j(t)}(c) -
      \alpha^0_{t'_j}(c)\right)\Delta_j\right) \right\|^2 \\
&   \ \ \ \ \ \ \ \ \ \  \ \ \ \ \ \ \ \ \ \  \ \ \ \ \ \ \ \ \ \ \ \  \ \ \ \ \ \ \ \ \ \   \ \ \ \ \ \
\ \ \ \ \ \ \ \ \  \ \ \ \ \ \ \ \ \ \ \ \ \ \   \text{(by
  a13), (\ref{ak0}) and
  a16))}\\
& = \left\| \sum_{j > i(k)} \sum_{l = j-1}^{j+1} \Delta_j\left(
      \alpha^{g_j(t)}_{h_j(t)}(c) -
      \alpha^0_{t'_j}(c)\right)^* \Delta_jv^2_t\Delta_l\left(
      \alpha^{g_l(t)}_{h_l(t)}(c) -
      \alpha^0_{t'_l}(c)\right)\Delta_l \right\| \\
& \ \ \ \ \ \ \ \ \ \ \ \ \ \ \ \ \ \ \ \ \ \ \ \ \ \ \ \ \ \ \ \  \ \
\ \ \ \ \ \ \ \ \ \ \ \  \
  \ \ \ \ \ \ \  \ \ \ \ \ \ \ \ \ \ \ \ \ \ \ \  \ \ \ \ \ \ \  \ \ \
  \ \ \ \ \ \ \ \ \ \text{(using a1))} \\
& \\
&\leq 12 \sup_{s\in[1,\infty)}
  \left\|\overline{\alpha}_s(c)\right\|^2\sup_{j \geq i(k)}
  \left\|\Delta_jv_t\right\| \ \ \ \ \ \ \ \ \ \ \text{(by Lemma 3.1 of \cite{MT2}).}
 \end{split}
\end{equation*}
In this way it follows from
(\ref{vdelta}) that $v_t\left(\psi_t(c) - \psi^0(c)\right) \sim
0$. The same arguments work to show that also $\left(\psi_t(c) -
  \psi^0(c)\right)v_t \sim 0$, giving us (\ref{A1}).

Note that (\ref{A1}) combined with a10) implies that
\begin{equation}\label{A2}
\left[v_t,\psi_t(c)\right] \sim 0
\end{equation}
and combined with a11) that
\begin{equation}\label{A3}
v_t(1 -v_t)\psi_t(c) \sim v_t(1-v_t)\mu^{11}_t(c)
\end{equation}
 for all $c \in A$.
Using (\ref{A2}), (\ref{A3}) and a8) we find that
\begin{equation}\label{step0}
\begin{split}
&\lambda^{11}_t(a)\lambda^{11}_t(b)  \\
& \sim v_t^4
\mu^{11}_t(a)\mu^{11}_t(b) + 2v_t^2\left(1 -
  v_t^2\right)\mu^{11}_t(a)\mu^{11}_t(b) + \left(1 -
  v_t^2\right)^2\psi_t(a)\psi_t(b)\\
& = \left(2v_t^2 - v_t^4\right)\mu^{11}_t(a)\mu^{11}_t(b) + \left(1 -
  v_t^2\right)^2\psi_t(a)\psi_t(b) .
\end{split}
\end{equation}
To continue we show next that
\begin{equation}\label{ogve}
(1-v_t) \psi_t(a)\psi_t(b) \sim (1-v_t) \psi_t(ab)
\end{equation}
for all $a,b \in A$. To this end let $t \in [k,k+1]$ and $a,b \in
F_k$. By using Lemma 3.1 from \cite{MT2} several times we find that
\begin{equation*}
\begin{split}
&(1 -v_t)\psi_t(a)\psi_t(b)  \\
& = (1-v_t) \sum_{j > t} \left[\sum_{l=j-1}^{j+1} \Delta_j
\alpha^{g_j(t)}_{h_j(t)}(a)\Delta_j\Delta_l
\alpha^{g_l(t)}_{h_l(t)}(b)\Delta_l \right] \\
& \ \ \ \ \ \ \ \ \ \ \ \ \ \ \ \ \ \ \ \ \ \ \ \ \ \ \ \ \ \ \ \ \ \
\ \ \ \ \ \ \ \ \ \ \ \ \ \ \ \ \ \ \ \ \ \ \ \ \ \ \ \ \ \ \ \ \ \ \
\ \ \ \ \ \ \ \ \ \
\text{(using a7))}\\
& = (1-v_t) \sum_{t < j \leq a_k} \left[\sum_{l=j-1}^{j+1} \Delta_j
\alpha^{g_j(t)}_{h_j(t)}(a)\Delta_j\Delta_l
\alpha^{g_l(t)}_{h_l(t)}(b)\Delta_l \right] \\
& \ \ \ \ \ \ \ \ \ \ \ \ \ + (1-v_t) \sum_{j > a_k} \left[\sum_{l=j-1}^{j+1} \Delta_j
\alpha^{1}_{h_j(t)}(a)\Delta_j\Delta_l
\alpha^{1}_{h_l(t)}(b)\Delta_l \right]  \\
& \ \ \ \ \ \  \ \ \ \ \ \ \ \ \ \ \ \ \ \ \  \ \ \ \ \ \ \ \ \ \ \ \
\ \ \ \ \ \ \ \ \ \ \ \ \ \ \ \ \   \ \ \ \ \ \ \ \ \ \ \ \ \ \ \ \ \
\ \ \ \ \ \ \ \ \ \ \ \ \ \ \ \ \text{(using (\ref{ak}))}\\
& \sim  (1-v_t) \sum_{ t < j \leq a_k} \left[\sum_{l=j-1}^{j+1} \Delta_j
\alpha^{g_j(t)}_{h_j(t)}(a)\Delta_j\Delta_l
\alpha^{g_j(t)}_{h_l(t)}(b)\Delta_l \right] \\
& \ \ \ \ \ \ \ \ \ \ \ \ \ + (1-v_t) \sum_{j > a_k} \left[\sum_{l=j-1}^{j+1} \Delta_j
\alpha^{1}_{h_j(t)}(a)\Delta_j\Delta_l
\alpha^{1}_{h_l(t)}(b)\Delta_l \right]  \\
& \ \ \ \ \ \  \ \ \ \ \ \ \ \ \ \
\ \ \ \ \ \ \ \ \ \ \ \ \ \ \ \ \   \ \ \ \ \ \ \ \ \ \ \ \ \ \ \ \ \
\ \ \ \ \ \ \ \
\ \ \ \ \ \ \ \ \ \ \  \text{(using (\ref{alpha}), a19) and a15) )}
\end{split}
\end{equation*}
\begin{equation*}
\begin{split}
& \sim  (1-v_t) \sum_{ t < j \leq a_k} \left[\sum_{l=j-1}^{j+1} \Delta_j
\alpha^{g_j(t)}_{h_j(t)}(a)\Delta_j\Delta_l
\alpha^{g_j(t)}_{h_j(t)}(b)\Delta_l \right] \\
& \ \ \ \ \ \ \ \ \ \ \ \ \ + (1-v_t) \sum_{j > a_k}\sum_{l=j-1}^{j+1} \Delta_j
\alpha^{1}_{h_j(t)}(a)\Delta_j\Delta_l
\alpha^{1}_{h_j(t)}(b)\Delta_l   \\
& \ \ \ \ \ \  \ \ \ \ \ \ \ \ \ \
\ \ \ \ \ \ \ \ \ \ \ \ \ \ \ \ \   \ \ \ \ \  \ \ \ \  \ \ \text{(using a17) and the
  uniform continuity of $\overline{\alpha}$)}\\
& \sim  (1-v_t) \sum_{ t < j \leq a_k} \Delta_j
\alpha^{g_j(t)}_{h_j(t)}(a)\alpha^{g_j(t)}_{h_j(t)}(b) \Delta_j + (1-v_t) \sum_{j > a_k} \Delta_j
\alpha^{1}_{h_j(t)}(a)
\alpha^{1}_{h_j(t)}(b)\Delta_j   \\
& \ \ \ \ \ \ \ \ \ \ \ \ \ \ \ \ \ \ \ \ \ \ \ \ \ \ \ \  \ \ \ \ \ \ \ \ \ \
\ \ \ \ \ \ \ \ \ \ \ \ \ \ \ \ \ \ \  \ \ \ \ \ \ \ \ \ \ \ \ \ \ \ \
\text{(using (\ref{enu23}), a19) and a3))}\\ 
& \sim (1-v_t) \sum_{ t < j \leq a_k} \Delta_j
\alpha^{g_j(t)}_{h_j(t)}(ab)\Delta_j + (1-v_t) \sum_{j > a_k} \Delta_j
\alpha^{1}_{h_j(t)}(ab)\Delta_j \ \ \  \ \ \ \ \ \ \ \ \text{(using a7))}\\
& = (1-v_t)\psi_t(ab)  \ \ \ \ \ \ \ \ \ \ \ \ \ \ \
\ \ \ \ \ \ \ \ \ \ \ \ \ \ \ \ \ \ \text{(using (\ref{enu24}))},
\end{split}
\end{equation*}
giving us (\ref{ogve}). Inserting (\ref{ogve}) into (\ref{step0}) we find that

\begin{equation*}\label{step1}
\begin{split}
&\lambda^{11}_t(a)\lambda^{11}_t(b)
 \sim \left(2v_t^2 -v_t^4\right)
\mu^{11}_t(a)\mu^{11}_t(b) + \left(1 -
  v_t^2\right)^2\psi_t(ab)\\
& \sim \left(2v_t^2 -v_t^4\right)
\left(\mu^{11}_t(ab) - \mu^{12}_t(a)\mu^{21}_t(b)\right) + \left(1 -
  v_t^2\right)^2\psi_t(ab) \\
& \ \ \ \ \ \ \ \ \ \ \ \ \ \ \ \ \ \ \ \ \ \ \ \ \ \  \ \ \ \ \ \ \ \
\ \ \ \ \ \ \ \ \ \ \ \ \ \ \ \ (\text{since $\mu$ is an
  asymptotic homomorphism})\\
& \sim \left(2v_t^2 -v_t^4\right)\mu^{11}_t(ab)  -  \mu^{12}_t(a)\mu^{21}_t(b) + \left(1 -
  v_t^2\right)^2\psi_t(ab) \\
& \ \ \ \ \ \ \ \ \ \ \ \ \ \ \ \ \ \ \ \ \ \ \ \ \ \ \  \ \ \ \ \ \ \ \ \ \ \ \ \ \ \ \ \ \
\ \ \ \ \ \ \ \ \ \ \ \ \ \ \ \ \ \ \ \ \ \ \ \ \ \ \ \ \ \ \ \ \ \ \
\ \ \ \ \ \ \ \ \ \ \ (\text{using a9)})\\
& = v_t^2\mu^{11}_t(ab) + \left(v_t^2 - v_t^4\right)\mu^{11}_t(ab) -  \mu^{12}_t(a)\mu^{21}_t(b) + \left(1 -
  v_t^2\right)^2\psi_t(ab) \\
& \sim  v_t^2\mu^{11}_t(ab)+ \left(v_t^2 -v_t^4\right)\psi_t(ab) -  \mu^{12}_t(a)\mu^{21}_t(b) + \left(1 -
  v_t^2\right)^2\psi_t(ab) \\
& \ \ \ \ \ \ \ \ \ \ \ \ \ \ \ \ \ \ \ \ \ \ \ \ \ \ \ \ \ \ \ \ \ \
\ \ \ \ \ \ \ \ \ \ \ \ \ \ \ \ \ \ \ \ \ \ \ \ \ \ \ \ \ \ \ \ \ \ \
\ \ \
\ \ \ \ \ \ \ \ \ \ \ \ \ \ \ \ \ \ (\text{using (\ref{A3})}) \\
& = v_t^2\mu^{11}_t(ab) + \left(1-v_t^2\right)\psi_t(ab) - \mu^{12}_t(a)\mu^{21}_t(b)\\
& \sim  v_t\mu^{11}_t(ab)v_t  + \left(1 -
  v_t^2\right)^{\frac{1}{2}}\psi_t(ab)\left(1 -
  v_t^2\right)^{\frac{1}{2}} -  \mu^{12}_t(a)\mu^{21}_t(b)\\
& \ \ \ \ \ \ \ \ \ \ \ \ \ \ \ \ \ \ \ \ \ \ \ \ \ \ \ \ \ \ \ \ \ \
\ \ \ \ \ \ \ \ \ \ \ \ \ \ \ \ \ \ \ \ \ \ \ \ \ \ \ \ \ \ \ \ \ \ \
\ \  \ \text{ (using a8) and (\ref{A2})),}
\end{split}
\end{equation*}
which gives us (\ref{tjek2}).

\qed

\subsection{On the dependence of the folding on the folding data}

We can now show that if a folding $\varphi_f$ of $\varphi$ is
semi-invertible then the same is true for any other folding of
$\varphi$ and that the unitary equivalence class of $\varphi_f$, modulo
asymptotically split extensions does not depend on the folding data.
The main step is the following lemma.

\begin{lemma}\label{crux2} Let $A$ and $B$ be $C^*$-algebras, $A$
  separable, $B$ stable and $\sigma$-unital. Let $\varphi : A \to
Q(B)$ be an asymptotic extension, and $f =\left( \overline{\varphi},
  \left\{u_n\right\}, \left\{t_n\right\}\right)$, $f'= \left(
  \widehat{\varphi}, \left\{u'_n\right\}, \left\{t'_n\right\} \right)$
two sets of folding data for $\varphi$. Assume that $\psi : A \to
Q(B)$ is an extension such that $\varphi_f \oplus \psi \oplus 0$ is
asymptotically split.

It follows that $\varphi_{f'} \oplus \psi \oplus 0$ is asymptotically
split.

\end{lemma}

\begin{proof} Thanks to Lemma \ref{uniform7} we may assume
  that $\varphi$, $\overline{\varphi}$ and $\widehat{\varphi}$ are all
  uniformly continuous. Since $\lim_{n \to \infty} t_n = \lim_{n \to \infty}
  t'_n =\infty$ we can use
  Lemma \ref{deltaepsilon} recursively to obtain a unit sequence
  $\left\{u''_n\right\}$ and increasing functions $m, m' : \mathbb N \to \mathbb N$ such that
  $m(0) = m'(0) = 0$,
\begin{enumerate}
\item[a20)] \label{enu310} $t_{m(n)} \geq n$,
\item[a21)] \label{enu311} $u''_nu_{m(n)} = u_{m(n)}$,
\item[a22)] \label{enu312} $u_{m(n)}u''_{n-1} = u''_{n-1}$ for all $n \geq 1$, and
\item[a23)] \label{enu313} ${t_n -t_{n-1}} \leq {t_{m(n)} - t_{m(n-1)}}$, $n \geq 1$,
\end{enumerate}
and also
\begin{enumerate}
\item[a24)] \label{enu314} $t'_{m'(n)} \geq n$,
\item[a25)] \label{enu315} $u''_nu'_{m'(n)} = u'_{m'(n)}$,
\item[a26)] \label{enu316} $u'_{m'(n)}u''_{n-1} = u''_{n-1}$, $n \geq 1$, and
\item[a27)] \label{enu317} ${t'_n -t'_{n-1}} \leq {t'_{m'(n)} -
    t'_{m'(n-1)}}$, $n \geq 1$.
\end{enumerate}
In addition we arrange that $g =\left(\overline{\varphi},
  \left\{u''_n\right\}, \left\{t_n\right\} \right)$ is folding data
for $\varphi$ and that
\begin{equation}\label{inaddition}
\lim_{n \to \infty}  \sup_{1 \leq t \leq n+3} \left\|(1-u''_n)
    \left(\overline{\varphi}_t(a) -
      \widehat{\varphi}_t(a)\right)\right\| = 0
\end{equation}
for all $a \in A$. Then also $g' = \left(\overline{\varphi},
  \left\{u_n''\right\}, \left\{t_n'\right\}\right)$ is folding data
for $\varphi$. We claim that there are Lipschitz
re-parametrisations $r,r' : [1,\infty) \to [1,\infty)$
  such that
\begin{enumerate}
\item[a28)] \label{enu321} $r(n) \leq n$ for all $n \in \mathbb N$ and
  \item[a29)] \label{enu323} $\varphi_{g} \oplus 0$ is unitarily
  equivalent to $\varphi_{h} \oplus 0$, where $h =
  \left(\overline{\varphi}, \{u_n\}, \left\{r\left(t_n\right)\right\}\right)$,
\end{enumerate}
and
\begin{enumerate}
\item[a30)] \label{enu324} $r'(n) \leq n$ for all $n \in \mathbb N$ and
  \item[a31)] \label{enu326} $\varphi_{g'} \oplus 0$ is unitarily
  equivalent to $\varphi_{h'} \oplus 0$, where $h' =
  \left(\widehat{\varphi}, \left\{u'_n\right\}, \left\{
      r'\left(t'_n\right) \right\} \right)$.
\end{enumerate}
The construction of $r$ and $r'$ are almost identical, but slightly
more demanding for $r'$ since we pass from $\overline{\varphi}$ to $\widehat{\varphi}$. We
describe therefore only the construction of $r'$.

Define $r' : [1,\infty) \to [1,\infty)$ to be the continuous
function such that $r'$ is linear on $\left[t'_{m(n-1)},t'_{m(n)}
\right]$, $r'\left(t'_{m(n-1)}\right) = t'_{n-1}$ and
$r'\left(t'_{m(n)}\right) = t'_n$. Then $r'$ is Lipschitz thanks
to a27) and $\lim_{t \to \infty} r'(t) = \infty$ since
$\lim_{n \to \infty} t'_n = \infty$. Using a24) we find
that
$$
r'^{-1}\left([1,n]\right) \supseteq r'^{-1}\left([1,t'_n]\right)
\supseteq \left[1,t'_{m(n)}\right] \supseteq [1,n] .
$$
It follows that $r'(n) \leq n$ for all $n$. Furthermore, it is straightforward
to combine (\ref{inaddition}) with the fact that $g$ is folding data
for $\varphi$ to verify that the same is true for $g'$. (Recall that
$\widehat{\varphi}$ is uniformly continuous.) It remains now only to show that $\varphi_{g'} \oplus 0$ is unitarily equivalent to $\varphi_{h'} \oplus
0$. For this purpose note first that by Lemma \ref{alt} we may as well show that $\varphi^{g'} \oplus 0$ is unitarily
equivalent to $\varphi^{h'} \oplus 0$. This is done as
follows.

Set $\Delta'_0 = \sqrt{u'_0}, \ \Delta'_n = \sqrt{u'_n -u'_{n-1}}, n \geq 1$,
$\Delta''_0 = \sqrt{u''_0}, \ \Delta''_n = \sqrt{u''_n -
  u''_{n-1}}$. It follows from a25) and a26) that
\begin{equation}\label{zero}
m'(n-1) < k \leq m'(n) \ \Rightarrow \ \Delta'_k\Delta''_{j} = 0, \ j \notin \{n-1,n\} .
\end{equation}
Set $V = \sum_{i,j} \Delta''_i\Delta'_j \otimes e_{ij}$
which is a partial isometry in $M(B \otimes \mathbb K)$. By using
(\ref{zero}) and Lemma 3.1 of \cite{MT2} we find that
\begin{equation*}\label{oooq}
\begin{split}
&V\widehat{\varphi}^{h'}(a)V^* = \sum_{i,j} \sum_{k=0}^{\infty}
\sum_{l=k-1}^{k+1} \Delta''_i {\Delta'_k}^2
\widehat{\varphi}_{r'\left(t'_k\right)}(a){\Delta'_l}^2\Delta''_j \otimes e_{ij} \\
& = \sum_{i=0}^{\infty} \sum_{j=i-3}^{i+3} \sum_{k=0}^{\infty}\sum_{l=k-1}^{k+1}
 \Delta''_i {\Delta'_k}^2
\widehat{\varphi}_{r'\left(t'_k\right)}(a){\Delta'_l}^2\Delta''_j \otimes e_{ij} \\
& = \sum_{i=0}^{\infty} \sum_{j=i-3}^{i+3} \sum_{k=0}^{\infty}\sum_{l=k-1}^{k+1}
 \Delta''_i {\Delta'_k}^2
\widehat{\varphi}_{t'_i}(a){\Delta'_l}^2\Delta''_j \otimes e_{ij} \ \ \ \ \ \
\text{modulo $B \otimes \mathbb K$.}
\end{split}
\end{equation*}
It follows from (\ref{inaddition}) that
\begin{equation*}\label{oooq2}
\begin{split}
&\sum_{i=0}^{\infty} \sum_{j=i-3}^{i+3} \sum_{k=0}^{\infty}\sum_{l=k-1}^{k+1}
 \Delta''_i {\Delta'_k}^2
\widehat{\varphi}_{t'_i}(a){\Delta'_l}^2\Delta''_j \otimes e_{ij} \\
& = \sum_{i=0}^{\infty} \sum_{j=i-3}^{i+3} \sum_{k=0}^{\infty}\sum_{l=k-1}^{k+1}
 \Delta''_i {\Delta'_k}^2
\overline{\varphi}_{t'_i}(a){\Delta'_l}^2\Delta''_j \otimes e_{ij} \ \ \ \ \ \
\text{modulo $B \otimes \mathbb K$.}
\end{split}
\end{equation*}
Note that
$$
\sum_{k=0}^{\infty}\sum_{l=k-1}^{k+1}
 \Delta''_i {\Delta'_k}^2
\overline{\varphi}_{t'_i}(a){\Delta'_l}^2 = \sum_{k \geq m'(i-1)}\sum_{l=k-1}^{k+1}
 \Delta''_i {\Delta'_k}^2
\overline{\varphi}_{t'_i}(a){\Delta'_l}^2,
$$
thanks to (\ref{zero}). Since
\begin{equation*}
\begin{split}
&\left\|\sum_{k\geq m'(i-1)}^{\infty}\sum_{l=k-1}^{k+1}
  {\Delta'_k}^2
\left[\overline{\varphi}_{t'_i}(a),{\Delta'_l}^2\right]\right\|\\
& \leq \left\|\sum_{k\geq m'(i-1)}{\Delta'_k}^2 \right\|^{\frac{1}{2}} \left\|
  \sum_{k\geq m'(i-1)} \Delta'_k\left[\overline{\varphi}_{t'_i}(a),
    \sum_{l=k-1}^{k+1} {\Delta'_l}^2 \right]\left[\overline{\varphi}_{t'_i}(a),
    \sum_{l=k-1}^{k+1} {\Delta'_l}^2 \right]^*\Delta'_k
\right\|^{\frac{1}{2}}\\
& \leq \left\|
  \sum_{k\geq m'(i-1)} \Delta'_k\left[\overline{\varphi}_{t'_i}(a),
    \sum_{l=k-1}^{k+1} {\Delta'_l}^2 \right]\left[\overline{\varphi}_{t'_i}(a),
    \sum_{l=k-1}^{k+1} {\Delta'_l}^2 \right]^*\Delta'_k,
\right\|^{\frac{1}{2}}
\end{split}
\end{equation*}
it follows from the compatibility of $\left(\{u'_n\},\{t'_n\}\right)$
with $\overline{\varphi}$ and Lemma 3.1 of \cite{MT2} that
\begin{equation*}
\begin{split}
& \sum_{i=0}^{\infty} \sum_{j=i-3}^{i+3} \sum_{k=0}^{\infty}\sum_{l=k-1}^{k+1}
 \Delta''_i {\Delta'_k}^2
\overline{\varphi}_{t'_i}(a){\Delta'_l}^2\Delta''_j \otimes e_{ij} \\
&=  \sum_{i=0}^{\infty} \sum_{j=i-2}^{i+3} \sum_{k=0}^{\infty}\sum_{l=k-1}^{k+1}
 \Delta''_i {\Delta'_k}^2{\Delta'_l}^2
\overline{\varphi}_{t'_i}(a)\Delta''_j \otimes e_{ij} \ \ \ \ \ \
\text{modulo $B \otimes \mathbb K$}\\
&= \sum_{i=0}^{\infty} \sum_{j=i-2}^{i+3}
 \Delta''_i
\overline{\varphi}_{t'_i}(a)\Delta''_j \otimes e_{ij} \ \ \ \ \ \ \ \
\ \ \ \ \ \ \ \
\text{(using a1) - a3) for $\left\{u_n'\right\}$)}\\
& = \overline{\varphi}^{g'}(a).
\end{split}
\end{equation*}
Since $V^*V \widehat{\varphi}^{h'}(a) =\widehat{\varphi}^{h'}(a)V^*V =\widehat{\varphi}^{h'}(a) $ modulo $B
\otimes \mathbb K$ it follows that $V^* \overline{\varphi}^{g'}(a)V =
\widehat{\varphi}^{h'}(a)$ modulo $B \otimes
\mathbb K$. Thus an application of Kasparov's stabilisation
theorem as in the proof of Lemma \ref{alt} shows that $V$ can be
dilated to give a unitary equivalence between $\varphi^{g'} \oplus 0$
and $\varphi^{h'} \oplus 0$. Hence a31) follows from Lemma \ref{alt}.

Using a28) - a31) we can now complete the proof as
follows: By assumption $\varphi_f \oplus \psi \oplus 0$ is
asymptotically split. Since $f$ and $h$ only differ in the
discretization sequences, it follows from Lemma \ref{vladlemma} that
$\varphi_h \oplus \psi \oplus 0$ is asymptotically split. Then
(\ref{enu323}) implies that so is $\varphi_g \oplus \psi \oplus
0$. $g$ and $g'$ differ also only in the discretization sequence so
another application of Lemma \ref{vladlemma} shows that also
$\varphi_{g'}\oplus \psi \oplus 0$ is asymptotically split. Hence a31) shows that $\varphi_{h'} \oplus \psi \oplus 0$ is
asymptotically split and a final application of Lemma \ref{vladlemma}
implies then that the same is true for $\varphi_{f'} \oplus \psi
\oplus 0$.
\end{proof}



We can now combine Lemma \ref{crux2} with Lemma \ref{vladlemma} to
obtain the following:

\begin{prop}\label{removefold} Let $A$ and $B$ be $C^*$-algebras, $A$
  separable, $B$ stable and $\sigma$-unital. Let $\varphi, \varphi' : A \to Q(B)$ be
  asymptotic extensions which are strongly homotopic. Let $f$ and $f'$
  be any folding data for $\varphi$ and $\varphi'$, respectively. Assume that $\psi : A \to Q(B)$ is an extension such that $\varphi_f \oplus
\psi \oplus 0$ is asymptotically split.

It follows that $\varphi'_{f'}
\oplus \psi \oplus 0$ is asymptotically split.
\begin{proof} Let $f = \left(\overline{\varphi}, \left\{u_n\right\},
    \left\{t_n\right\}\right)$ be folding data for $\varphi$ and $f' = \left(\overline{\varphi'}, \left\{u'_n\right\},
    \left\{t'_n\right\}\right)$ for $\varphi'$. It follows from Lemma \ref{uniform7} that we can assume
  that $\varphi$ and $\varphi'$ are uniformly continuous. Let $\alpha
  : A \to C[0,1] \otimes Q(B)$ be a strong homotopy connecting
  $\varphi$ and $\varphi'$. By Lemma 4.3 of \cite{MT2} (or Lemma
  \ref{uniform7} above) there is a
  Lipschitz re-parametrisation $r : [1,\infty) \to [1,\infty)$ such that
  $\alpha^r$ is uniformly continuous. Since $r$ is Lipschitz there are strong
  homotopies consisting of uniformly continuous asymptotic
  homomorphisms $A \to C[0,1]\otimes Q(B)$ connecting $\varphi$ to
  $\varphi^r$ and $\varphi'$ to ${\varphi'^r}$. Concatenation with
  $\alpha^r$ gives us a strong homotopy which connects $\varphi$ to
  $\varphi'$ and consists of a uniformly continuous asymptotic
  homomorphism $A \to C[0,1]\otimes Q(B)$. The desired conclusion
  follows then by combining Lemma \ref{crux2} with Lemma \ref{vladlemma}.
\end{proof}
\end{prop}

\subsection{The pairing}

To obtain the desired pairing between extensions and asymptotic
homomorphisms we must review the composition product of asymptotic
homomorphisms, as defined by Connes and Higson in \cite{CH}, in a form
suitable for the present purpose.

\begin{lemma}\label{pair0} Let $A,A'$  and $D$ be $C^*$-algebras,
  $A,A'$ separable. Let $\varphi : A \to D$ and $\lambda : A' \to A$ be
  equi-continuous asymptotic homomorphisms. Let $X \subseteq A'$ be a
  $\sigma$-compact subset with dense span in $A'$.

There is a re-parametrisation $s : [1,\infty) \to [1,\infty)$ and
an equi-continuous family of maps $\kappa_{t,x} : A' \to D, t,x
\in [1,\infty)$, $t \geq x$, such that
\begin{enumerate}
\item[i)] $\lim_{x \to \infty} \sup_{t \geq x} \left\| \varphi_t \circ
  \lambda_{s(x)}(a) - \kappa_{t,x}
(a) \right\| = 0$ for all $a \in X$, and
\item[ii)] $\lim_{x \to \infty} \sup_{t \geq x} \left\| \kappa_{t,x}(a)\kappa_{t,x}(b)
  - \kappa_{t,x}(ab)\right\| = 0$,
\item[iii)] $\lim_{x \to \infty} \sup_{t \geq x} \left\| \kappa_{t,x}(a) + z\kappa_{t,x}(b)
  - \kappa_{t,x}(a + z b)\right\| = 0$,
\item[iv)] $\lim_{x \to \infty} \sup_{t \geq x} \left\| \kappa_{t,x}(a^*)
  - \kappa_{t,x}(a)^*\right\| = 0$
\item[v)] $\sup_{t,x} \left\|\kappa_{t,x}(a)\right\| < \infty$
\end{enumerate}
for all $a,b \in A'$ and all $z \in \mathbb C$.
\end{lemma}
\begin{proof} By the method used to define the composition product of
  $\varphi$ and $\lambda$ in \cite{CH} we get a re-parametrisation $r
  : [1,\infty) \to [1,\infty)$ such that $r(1) = 1$ and
\begin{enumerate}
\item[a32)] $\limsup_{ t \to \infty} \sup_{y \geq r(t)} \left\| \varphi_{y} \circ
  \lambda_t(a) - \varphi_{y} \circ
  \lambda_t(b) \right\| \leq \left\|a-b\right\|$,
\item[a33)]  $\lim_{t \to \infty} \sup_{y \geq r(t)} \left\| \varphi_{y} \circ
  \lambda_t(a) \varphi_{y} \circ
  \lambda_t(b) - \varphi_{y} \circ
  \lambda_t(ab)\right\| = 0$,
\item[a34)]  $\lim_{t \to \infty} \sup_{y \geq r(t)} \left\| \varphi_{y} \circ
  \lambda_t(a) + z\varphi_{y} \circ
  \lambda_t(b) - \varphi_{y} \circ
  \lambda_t(a+z b)\right\| = 0$,
\item[a35)]  $\lim_{t \to \infty} \sup_{y \geq r(t)} \left\| \varphi_{y} \circ
  \lambda_t(a^*)  - \varphi_{y} \circ
  \lambda_t(a)^*\right\| = 0$
\end{enumerate}
for all $a,b \in X$ and all $z \in \mathbb C$. Then $s =
r^{-1} : [1,\infty) \to [1,\infty)$ is a re-parametrisation such that
\begin{enumerate}
\item[a36)]  $\limsup_{ x \to \infty} \sup_{t \geq x} \left\| \varphi_{t} \circ
  \lambda_{s(x)}(a) - \varphi_{t} \circ
  \lambda_{s(x)}(b) \right\| \leq \left\|a-b\right\|$,
\item[a37)] $\lim_{x \to \infty} \sup_{t \geq x} \left\| \varphi_{t} \circ
  \lambda_{s(x)}(a) \varphi_{t} \circ
  \lambda_{s(x)}(b) - \varphi_{t} \circ
  \lambda_{s(x)}(ab)\right\| = 0$,
\item[a38)]  $\lim_{x \to \infty} \sup_{t \geq x}  \left\| \varphi_{t} \circ
  \lambda_{s(x)}(a) + z\varphi_{t} \circ
  \lambda_{s(x)}(b) - \varphi_{t} \circ
  \lambda_{s(x)}(a+ zb)\right\| = 0$,
\item[a39)]  $\lim_{x \to \infty} \sup_{t \geq x} \left\| \varphi_{t} \circ
  \lambda_{s(x)}(a^*)  - \varphi_{t} \circ
  \lambda_{s(x)}(a)^*\right\| = 0$
\end{enumerate}
for all $a,b \in X$ and all $z \in \mathbb C$. Set
$Z = \left\{ (t,x) \in [1,\infty)^2 : \ t \geq x \right\}$,
$\mathcal A = C_b\left(Z,D\right)$
and
$$
\mathcal J = \left\{ f \in \mathcal A : \ \lim_{x \to \infty} \sup_{t
    \geq x}\left\|f(t,x)\right\| = 0 \right\} .
$$
Then $\mathcal J$ is an ideal in $\mathcal A$ and we let $q : \mathcal
A \to \mathcal A/\mathcal J$ be the quotient map. It follows from
a36)-a39) that there is a $*$-homomorphism $\Phi : A' \to \mathcal
A/\mathcal J$ such that $\Phi(a) = q(f_a)$ for each $a \in X$,
where $f_a \in \mathcal A$ is defined such that $f_a(t,x) =
\varphi_t\circ \lambda_{s(x)}(a)$. By the Bartle-Graves selection
theorem there is a continuous right-inverse $S : \mathcal A/\mathcal J
\to \mathcal A$ for $q$. Set $\kappa_{t,x}(a) = S\circ
\Phi(a)(t,x)$. Then i)-iv) hold.
 \end{proof}

With the re-parametrisation $s$ from Lemma \ref{pair0} at hand we can now introduce the composition product
$\bullet : \ [[A,D]] \times [[A',A]] \to [[A',D]]$
of Connes and Higson, \cite{CH}, such that
$$
[\varphi] \bullet [\lambda] = [\Phi],
$$
where $\Phi : A' \to
D$ is any equi-continuous asymptotic homomorphism with the property
that
$$
\lim_{t \to \infty} \Phi_t(a) - \kappa_{t,s'(t)}(a) = 0
$$ for all
$a \in X$ and $s'$ is any re-parametrisation for which $s'
\leq s$.



\begin{lemma}\label{pair1} Let $A,A'$ be separable $C^*$-algebra, $B$
  stable and $\sigma$-unital. Let $\varphi : A \to Q(B)$ be a
  semi-invertible extension and $\lambda : A' \to A$ an asymptotic
  homomorphism.

It follows that any folding $\left(\varphi \circ
    \lambda\right)_f$ of the asymptotic extension $\varphi \circ \lambda$
  is semi-invertible.
\begin{proof} Let $\varphi' : A \to Q(B)$ be an extension such that
  $\varphi \oplus \varphi'$ is asymptotically split. By considering
  the composition product between $\lambda$ and an asymptotic lift of
  $\varphi \oplus \varphi'$ it follows that there is a re-parametrisation $s$ such that
  $\left(\varphi \circ \lambda^s\right)  \oplus \left(\varphi' \circ
    \lambda^s\right)$ is an asymptotic extension which is asymptotically split in the sense of
  \cite{MT2}. It follows therefore from Lemma 4.4 of \cite{MT2} that
  $\left(\varphi \circ \lambda^s\right)_f$ is semi-invertible for any
  folding $\left(\varphi \circ \lambda^s\right)_f$ of $\varphi \circ
  \lambda^s$. Since $\varphi \circ \lambda^s$ is strongly homotopic to
  $\varphi \circ \lambda$, it follows from Proposition
  \ref{removefold} that the same is true for any folding of $\varphi
  \circ \lambda$.
\end{proof}
\end{lemma}

As in \cite{MT3},\cite{MT4} and \cite{MT5} we denote
$\Ext^{-\frac{1}{2}}(A,B)$ the group of semi-invertible extensions of
$A$ by $B$ modulo unitary equivalence and addition by asymptotically
split extensions. By combining Proposition \ref{removefold} and Lemma \ref{pair0} with
Lemma 4.5 of \cite{MT2} we get the desired pairing:

\begin{thm}\label{pairingthm} Let $A,A'$ be separable $C^*$-algebra, $B$
  stable and $\sigma$-unital. There is map a
$$
\star : \Ext^{-\frac{1}{2}}(A,B) \times [[A',A]]   \to
\Ext^{-\frac{1}{2}}(A',B)
$$
such that $[\varphi] \star [\lambda] =  \left[\left(\varphi \circ
    \lambda\right)_f\right]$ where $\left(\varphi \circ
    \lambda\right)_f$ is an arbitrary folding of the asymptotic extension
    $\varphi \circ \lambda$.
\end{thm}

In the construction of the pairing $\star$ we have fixed the
strictly positive element $b \in B$ which we used to define the
unit sequences. It follows from Lemma 4.4 of \cite{MT2} that
$\star$ is independent of the choice of $b$. The freedom in the
choice of strictly positive element makes it easy to show
that
\begin{equation}\label{additive}
\left([\varphi] + [\psi]\right) \star [\lambda] = \left([\varphi]
  \star [\lambda] \right) + \left([\psi] \star
  [\lambda] \right) .
\end{equation}

As should be expected it is somewhat more tricky to establish the
natural associativity involving $\star$ and the composition
product $\bullet$ for asymptotic homomorphisms.

\begin{lemma}\label{compo} Let $A',A,B$ be $C^*$-algebras, $A',A$
  separable and $B$ stable and $\sigma$-unital. Let $\varphi : A \to Q(B)$ be an
  asymptotic extension and $\lambda : A' \to A$ an asymptotic
  homomorphism. There is a folding $\varphi_f$ of $\varphi$, a
  re-parametrisation $s$, a
  folding $\left(\varphi_f \circ \lambda^s\right)_{f'}$ of $\varphi_f \circ
  \lambda^s$ and an asymptotic extension $\mu : A'  \to Q(B)$ such that
\begin{enumerate}
\item[i)] $[\mu] = \left[\varphi\right] \bullet \left[\lambda\right]$ in
  $[[A',Q(B)]]$ where $\bullet$ denotes the composition product
  $\bullet : \ [[A',A]] \times [[A,Q(B)]] \to [[A',Q(B)]]$, and
\item[ii)] $\mu_{f''} = \left(\varphi_f \circ \lambda^s\right)_{f'}$ for
  some folding $\mu_{f''}$ of $\mu$.
\end{enumerate}
\begin{proof} Let $F_1 \subseteq
  F_2 \subseteq F_3 \subseteq \dots $ be a sequence of finite subsets
  with dense union in $A'$ and set $X = \bigcup_n F_n$. By construction of the composition product there is a
  parametrisation $s$  and an
  equi-continuous family of maps
$\kappa_{t,x} : A \to Q(B), t,x \in [1,\infty), t\geq x,$
such that i)-v) of Lemma \ref{pair0} hold.
The composition product $[\varphi]\bullet [\lambda]$ is
then represented by any asymptotic extension $\mu$ with the property
that $\lim_{t \to \infty} \mu_t(a) - \varphi_t \circ
\lambda_{s(r(t))}(a) = 0$ for all $a \in X$, where $r$ can be any
re-parametrisation  such that $r(t) \leq t$ for all $t$.

From the
Bartle-Graves selection theorem we get an equi-continuous family
of maps
 $\overline{\kappa}_{t,x} : A \to M(B), \ t,x \in [1,\infty), t\geq x$,
such that $q_B \circ \overline{\kappa}_{t,x} = \kappa_{t,x}$ for all
$t,x$. Let $f = \left(\overline{\varphi}, \{u_n\},
    \{t_n\}\right)$ be folding data for $\varphi$ which have a series
  of additional properties which we now describe.  For each $k \in \mathbb N$ there is a
  $\delta_k > 0$ such that
\begin{equation}\label{disccond1}
 \max\{|s-s'|, |t-t'|\} \leq \delta_k \
\Rightarrow \ \left\|\overline{\kappa}_{t,s}(a) -
    \overline{\kappa}_{t',s'}(a)\right\| \leq \frac{1}{k}
\end{equation}
when $s,s',t,t' \in [1,k+2]$ and $a \in F_k$. We shall require of the discretization $\{t_n\}$
that
\begin{equation}\label{disccond2}
\left|t_i - t_{i+1}\right|  \leq \delta_k
\end{equation}
when $t_i \leq k$. Concerning the unit sequence $\{u_n\}$ we will
require that
\begin{equation}\label{f}
\lim_{n \to \infty} \sup_{t,x \in [1,n+2]} \left[\left\|(1-u_n)f(t,x)\right\| -
\left\|q_B\left(f(t,x)\right)\right\| \right] = 0
\end{equation}
when $f$ is any of the $M(B)$-valued functions
\begin{itemize}
\item $f(t,x) = \overline{\kappa}_{t,x}(a) \overline{\kappa}_{t,x}(b) -
    \overline{\kappa}_{t,x}(ab)$,
\item  $f(t,x) = \overline{\kappa}_{t,x}(a^*) -
  \overline{\kappa}_{t,x}(a)^*$,
\item  $f(t,x) = \overline{\kappa}_{t,x}(a+\lambda b) -
  \overline{\kappa}_{t,x}(a) - \lambda\overline{\kappa}_{t,x}(b)$
\end{itemize}
for any $a,b \in A'$, $\lambda\in\mathbb C$, or
\begin{itemize}
\item\label{f1} $f(t,x) = \overline{\varphi}_t \circ \lambda_{s(x)}(a) -
    \overline{\kappa}_{t,x}(a)$
\end{itemize}
for any $a \in X$, and also that
\begin{equation}\label{hug12}
\lim_{n \to \infty} \sup_{t,x \in [1,n+2]} \left\| u_n
  \overline{\kappa}_{t,x}(a) -
  \overline{\kappa}_{t,x}(a) u_n\right\| = 0
\end{equation}
for all $a \in A'$.


 We choose next a discretization
  $\{t'_n\}$ of $[1, \infty)$ such that
\begin{equation}\label{help}
\lim_{n \to \infty} \sup_{v \in [1,\infty)} \sup_{t \in \left[t'_n,t'_{n+1}\right]} \left\|
  \overline{\varphi}_v \circ \lambda_{s(t)}(a) -
  \overline{\varphi}_v \circ \lambda_{s(t'_n)}(a) \right\| = 0
\end{equation}
for all $a \in X$. This is done as follows: Let $n \geq 2$. By
compactness of
$$
\left\{\lambda_{s(t)}(a): \
  t\in [1,n+2], a \in F_n\right\}
$$
and equi-continuity of
$\overline{\varphi}$ there is a $\delta > 0$ such that $\sup_{v \in
  [1,\infty)} \left\|\overline{\varphi}_v (y) -
  \overline{\varphi}_v(z) \right\| \leq \frac{1}{n}$ when $z,y \in \left\{\lambda_{s(t)}(a): \
  t\in [1,n+2], a \in F_n\right\}$ and $\left\|y-z\right\| \leq
\delta$. Choose then a $\delta' \in ]0,1]$ such that
$\left\|\lambda_{s(t)}(a) - \lambda_{s(t')}(a)\right\| \leq \delta$
when $t,t' \in [1,n+1], a \in F_n$ and $|t-t'| \leq \delta'$. Arrange that
$\left|t'_{i+1} - t'_i\right| \leq \delta'$ when $t_i \in [n,n+1]$. Then
(\ref{help}) holds.

Subsequently we choose folding data $f' =
  \left(\overline{\psi}, \{u'_n\}, \{t'_n\}\right)$ for $\varphi_f
  \circ \lambda^s$ with the additional properties that
\begin{equation}\label{help1}
\lim_{n \to \infty}  \sup_{t \in [1,n+2]}\left\|\Delta'_{n+j} \left[\sum_{k=0}^{\infty} \Delta_k
  \overline{\varphi}_{t_k} \circ \lambda_{s\left(t\right)}(a) \Delta_k
   - \overline{\psi}_{t}(a)\right]\right\| = 0
\end{equation}
and
\begin{equation}\label{help2}
\sum_{n =1}^{\infty} \sup_{t \in [1,n+2]} \left\| \Delta'_{n+j} \left[\sum_{k=0}^{\infty}
  \Delta_k\overline{\varphi}_{t_k} \circ \lambda_{s(t)}(a)\Delta_k\right]  -\left[\sum_{k=0}^{\infty}
  \Delta_k\overline{\varphi}_{t_k} \circ \lambda_{s(t)}(a) \Delta_k\right] \Delta'_{n+j}
\right\|  < \infty
\end{equation}
for all $j \in \{-1,0,1\}, a \in X$. It follows from (\ref{help1}) and
(\ref{help2}) that
\begin{equation}\label{interm}
\left(\varphi_f \circ \lambda\right)_{f'}(a) = q_B \left(\sum_{n=0}^{\infty}
   \left[\sum_{j=0}^{\infty} \Delta_j
  \overline{\varphi}_{t_j} \circ \lambda_{s\left(t'_n\right)}(a)
  \Delta_j\right] {\Delta'_n}^2\right).
\end{equation}
for all $a \in X$. As in the proof of Lemma \ref{crux2}
we can also require of $\{u'_n\}$ that there is a strictly increasing
function $m : \mathbb N \to \mathbb N$ such that $m(0) = 0$ and
\begin{enumerate}
\item[a40)] \label{enu51} $t_{m(n)} \geq t'_{n+1}$, $n \geq 1$,
\item[a41)] \label{enu52} $u'_n u_{m(n)} = u_{m(n)}$,
\item[a42)] \label{enu53} $u_{m(n)}u'_{n-1} = u'_{n-1}$, $n \geq 1$, and
\item[a43)] \label{enu54} $t'_n - t'_{n-1} \leq t_{m(n)} - t_{m(n-1)}$, $n \geq 1$.
\end{enumerate}
Define $r : [1,\infty) \to [1,\infty)$ such that $r$ is linear on
$\left[t_{m(n-1)}, t_{m(n)}\right]$, $r\left(t_{m(n-1)}\right) =
t'_{n-1}$ and $r\left(t_{m(n)}\right) =
t'_{n}$ for all $n$. It follows from a40) that $r(t) \leq t$
for all $t$ and from a43) that
\begin{equation}\label{Lip1}
\left|r(t)-r(t')\right| \leq
\left|t-t'\right|
\end{equation}
for all $t,t' \in [1,\infty)$. From a41) and a42) we deduce that
\begin{equation*}
\begin{split}
& \sum_{n=0}^{\infty}  \left[\sum_{j=0}^{\infty} \Delta_j
  \overline{\varphi}_{t_j} \circ \lambda_{s\left(t'_n\right)}(a)
  \Delta_j\right] {\Delta'_n}^2 = \\
& \sum_{n=1}^{\infty} \sum_{m(n-1) \leq k < m(n)} \sum_{j \in
  \{-1,0\}}  \Delta_k
\overline{\varphi}_{t_k} \circ
\lambda_{s\left(t'_{n+j}\right)}(a) \Delta_k  {\Delta'_{n+j}}^2 .
 \\
\end{split}
\end{equation*}
Now we combine
(\ref{help}) and Lemma 3.1 of \cite{MT2} to conclude that
\begin{equation*}
\begin{split}
&\left(\varphi_f \circ \lambda\right)_{f'}(a) = q_B \left(\sum_{n=1}^{\infty} \sum_{m(n-1) \leq k < m(n)} \sum_{j \in
  \{-1,0\}}  \Delta_k  \
\overline{\varphi}_{t_k} \circ
\lambda_{s\left(r(t_k)\right)}(a)  \Delta_k{\Delta'_{n+j}}^2\right) \\
& = q_B \left( \left[ \sum_{k=0}^{\infty}
  \Delta_k \overline{\varphi}_{t_k} \circ
\lambda_{s\left(r(t_k)\right)}(a) \Delta_k \right]\sum_{n=0}^{\infty} {\Delta'_n}^2  \right) \\
& = q_B \left(\sum_{k=0}^{\infty}
  \Delta_k \overline{\varphi}_{t_k} \circ
\lambda_{s\left(r(t_k)\right)}(a) \Delta_k  \right) \\
\end{split}
\end{equation*}
for all $a \in X$. It follows from (\ref{f}) that
\begin{equation}\label{stop}
\left(\varphi_f \circ \lambda\right)_{f'}(a) =
q_B\left(\sum_{k=0}^{\infty}
  \Delta_k\overline{\kappa}_{t_k,r(t_k)}(a)\Delta_k\right)
\end{equation}
for all $a \in X$. Now note that it follows from (\ref{disccond1}),
(\ref{disccond2}) and (\ref{Lip1}) that
$$
\lim_{k \to \infty} \sup_{t \in \left[t_k,t_{k+1}\right]}
\left\|\overline{\kappa}_{t_k,r\left(t_k\right)}(a) -
  \overline{\kappa}_{t,r(t)}(a)\right\| = 0
$$
for all $a \in A'$. Combining this with (\ref{f}) and (\ref{hug12}) we can
conclude that $f'' = \left(
  \overline{\kappa}_{t,r(t)}, \{u_n\}, \{t_n\}\right)$ is folding data
for the asymptotic extension $\left\{\kappa_{t,r(t)}\right\}_{t \in
  [1,\infty)}$. Since (\ref{stop}) implies that $\left(\varphi_f \circ
  \lambda\right)_{f'} = \mu_{f''}$, where $\mu_t = q_B \circ
\overline{\kappa}_{t,r(t)}$, this completes the proof.
\end{proof}
\end{lemma}

\begin{thm}\label{associativity} Let $A'',A',A$ be a separable
  $C^*$-algebras and $B$ a stable $\sigma$-unital $C^*$-algebra. Let
  $\nu : A'' \to A'$ and $\lambda : A' \to A$ asymptotic homomorphisms
  and $\varphi : A \to Q(B)$ a semi-invertible extension. Then
$$
\left([\varphi] \star [\lambda]\right) \star [\nu] = [\varphi] \star \left([\lambda]
  \bullet [\nu]\right)
$$
in $\Ext^{-\frac{1}{2}}(A'',B)$.
\begin{proof} Apply Lemma \ref{compo} with $\varphi \circ \lambda$ in
  the role of $\varphi$ and $\nu$ in the role of $\lambda$.
\end{proof}
\end{thm}

\section{Semi-invertibility}

\begin{lemma}\label{compo2} Let $A',A,B$ be $C^*$-algebras, $A',A$
  separable and $B$ stable and $\sigma$-unital. Let $\varphi : A \to Q(B)$ be an
  asymptotic extension and $\lambda : A' \to A$ an asymptotic
  homomorphism. Let $\nu : A' \to Q(B)$ be an asymptotic extension
  such that $[\nu] = [\varphi] \bullet [\lambda]$ in
  $[[A',Q(B)]]$. Assume that $\varphi_g$ is semi-invertible for some
  folding $\varphi_g$ of $\varphi$.

Then $\nu_{g'}$ is semi-invertible for every folding $\nu_{g'}$ of
$\nu$.
\begin{proof} Let $\varphi_f$, $s$, $\mu$ and $f''$ be as in Lemma
  \ref{compo}. By assumption $\varphi_g$ is semi-invertible for some
  folding $\varphi_g$ of $\varphi$ and it follows then from Proposition
  \ref{removefold} that also
  $\varphi_f$ is semi-invertible. Thus $\mu_{f''}$ is semi-invertible
  by Lemma \ref{compo}. Since $\mu$ is strongly homotopic to $\nu$ it
  follows from Proposition \ref{removefold} that $\nu_{g'}$ is
  semi-invertible for any folding of $\nu$.
\end{proof}
\end{lemma}

With the following definition we try to cover the most general result
about automatic semi-invertibility which can be obtained from the pairing of
$\Ext^{-1/2}$ with asymptotic homomorphisms. It is inspired by three
sources. One is the paper by Dadarlat and Loring on unsuspended
E-theory, \cite{DL}, where homotopy symmetric $C^*$-algebras are
introduced. Another is the paper \cite{V} of
Voiculescu where the notion of homotopy domination is introduced and
the third is the work of Dadarlat \cite{D} where it is shown that shape-equivalence of
separable $C^*$-algebras is the same thing as equivalence in the asymptotic
homotopy category of Connes and Higson.

\begin{defn}
Let $A$ and $A'$ be $C^*$-algebras. Following the notation of
\cite{DL} we denote by $\left[\id_A\right]$ the element of
$[[A,A\otimes \mathbb K]]$ represented by the $*$-homomorphism $s(a) =
a \otimes e$ for some minimal non-zero projection $e$ in $\mathbb K$.
We say that $A$ is
\emph{homotopy symmetric relative to} $A'$ when there are asymptotic homomorphisms
$\lambda : A \to A'$, $\mu : A' \to A\otimes \mathbb K$ and $\psi: A
\to A \otimes \mathbb K$ such that
$$
\left[\id_A\right] + \left[\psi\right] = [\mu] \bullet [\lambda]
$$ in
$[[A,A\otimes \mathbb K]]$. When $\psi$ can be taken to be zero, we
say that $A$ is \emph{shape dominated} by $A'$.
\end{defn}

Thus $A$ is homotopy symmetric in the sense of Dadarlat and Loring if
and only if it is homotopy symmetric relative to $0$, and shape
domination generalises homotopy domination in the sense of Voiculescu.

\begin{thm}\label{THM} Let $A',A,B$ be $C^*$-algebras, $A',A$
  separable and $B$ stable and $\sigma$-unital. Assume that $A$ is
  homotopy symmetric relative to $A'$ and that all extensions of $A'$
  by $B$ are semi-invertible.

  It follows that all extensions of $A$ by $B$ are semi-invertible.

\begin{proof} By Lemma 4.3 of \cite{MT5} it suffices to show that all
  extensions of $A \otimes \mathbb K$ by $B$ are semi-invertible,
  i.e. we may assume that $A$ is stable. Then our assumptions imply
  that there are asymptotic homomorphisms $\lambda : A \to A'\otimes
  \mathbb K$, $\mu : A' \otimes \mathbb K \to A$ and $\psi : A \to A$
  such that
$\left[\varphi \oplus (\varphi \circ \psi)\right] = [\varphi \circ \mu] \bullet [\lambda]$ in $[[A,Q(B)]]$ for any
extension $\varphi : A \to Q(B)$. By assumption any folding of $\varphi \circ \mu$ is semi-invertible and
hence Lemma \ref{compo2} implies that $\varphi \oplus \left(\varphi
  \circ \mu\right)_f$ is semi-invertible for any folding $\left(\varphi
  \circ \mu\right)_f$ of $\varphi \circ \mu$. It follows that
$\varphi$ is semi-invertible.
\end{proof}
\end{thm}

\section{Relation to E-theory}

Recall that the $E$-theory of Connes and Higson, \cite{CH}, depends on
a fundamental construction, \emph{the Connes-Higson construction},
which produces asymptotic homomorphisms out of extensions. Since the
asymptotic homomorphism obtained from an asymptotically split extension is
homotopic to $0$, the Connes-Higson construction gives rise to a group homomorphism
\begin{equation}\label{CH??}
CH : \ \Ext^{-\frac{1}{2}}(A,B) \to [[SA,B]] .
\end{equation}
As shown in \cite{DL} the group $[[SA,B]]$ is isomorphic to the
$E$-theory group $E(A,SB)$. Thus $CH$ gives a direct relation between
$\Ext^{-1/2}$ and $E$-theory. It is unknown if $CH$ is always an
isomorphism, but we can now show that it is when $A$ is shape
dominated by another $C^*$-algebra $A'$, for example a nuclear
$C^*$-algebra, for which $CH : \Ext^{-1/2}(A',B) \to [[SA',B]]$ \emph{is} an isomorphism.

\begin{thm}\label{chcommutative} Let $A',A,B$ be separable $C^*$-algebras, $B$ stable. Let $\varphi : A \to
  Q(B)$ be a semi-invertible extension and $\lambda: A' \to A$ an
  asymptotic homomorphism. It follows that
$$
CH\left([\varphi] \star [\lambda]\right) = CH[\varphi] \bullet
[S\lambda]
$$
in $[[SA',B]]$, where $S\lambda : SA' \to SA$ is the suspension of $\lambda$.
\end{thm}
\begin{proof} We refer to \cite{CH} for the description of the
  Connes-Higson construction we shall use here. Let $\psi$ be a lift
  of $\varphi \circ \lambda$. Applying the same re-parametrisation to both
  $\psi$ and $\lambda$ we can arrange that $\psi$ is uniformly
  continuous, and still have that $q_B \circ \psi_t = \varphi
  \circ \lambda_t$ for all $t$, cf. Lemma \ref{uniform7}. Let $f = \left(\psi,\{u_n\},\{t_n\}\right)$ be folding
  data defining the folding $\left(\varphi \circ
    \lambda\right)_f$, and let $\overline{\varphi} : A \to M(B)$ be a
  continuous lift of $\varphi$.  Let $R$ be a countable dense subset of
  $C_0(0,1)$ and $X$ a countable dense subset of $A'$. Define $u_t \in B$, $t \in [n,n+1]$,
  such that $u_t = (t-n)u_{n+1} + (n+1-t)u_n$. Let $r$ be a re-parametrisation of
  $[1,\infty)$ such that
\begin{enumerate}
\item[a44)] \label{F1} $r(t) \leq t$ for all $t$,
\item[a45)] \label{F2} $\lim_{n \to \infty} r(n+1) -r(n) = 0$,
\item[a46)] \label{F3} $\lim_{t \to \infty} (1-u_t) \left( \psi_{r(t)}(a) -
    \overline{\varphi}\circ \lambda_{r(t)}(a)\right) = 0$ for all $a
  \in X$, and
\item[a47)] \label{F4} $\left[CH(\varphi) \bullet (S\lambda)\right]$ is represented
  in $[[SA',B]]$ by an asymptotic homomorphism $\Phi : SA' \to B$ such
  that $\lim_{t \to \infty} g(u_t)
  \overline{\varphi}\left(\lambda_{r(t)}(a)\right) - \Phi_t(g \otimes a) = 0$ for all $g \in
  R$ and all $a \in X$.
\end{enumerate}
It follows from a44) and a45) that $f' = \left(\psi, \{u_n\}, \left\{r(n)\right\}\right)$ is folding
data for $\varphi \circ \lambda$ since $\left(\psi, \{u_n\},
  \left\{t_n\right\}\right)$ is. If we let $A(t)
\sim B(t)$ mean that $\lim_{t \to \infty} A(t) -B(t) = 0$, we have for
any $g \in R$ and $a \in X$ that
\begin{equation*}
\begin{split}
& g(u_t)\left(\sum_{n=0}^{\infty} \Delta_n \psi_{r(n)}(a)\Delta_n
\right) = \sum_{j=n-4}^{n+4} g(u_t)\Delta_j \psi_{r(j)}(a)\Delta_j \ \
\ \ \ \text{where $t \in [n,n+1]$}\\
 & \ \ \ \ \ \ \ \ \ \ \ \ \  \ \ \ \ \ \ \ \ \ \ \ \ \ \ \ \ \ \ \ \ \
 \ \ \ \ \ \ \ \ \ \ \ \ \ \ \ \ \ \ \ \ \ \ \ \text{(by a1) and the definition of
   $\{u_t\}$)}
\end{split}
\end{equation*}
\begin{equation*}
\begin{split}
& \sim \sum_{j=n-4}^{n+4} g(u_t)\Delta_j \psi_{r(t)}(a)\Delta_j\\
& \ \ \ \ \ \ \ \ \ \ \ \ \  \ \ \ \ \ \ \ \ \ \ \ \ \ \ \ \ \ \ \ \ \
 \ \ \ \ \ \ \ \ \ \ \ \ \text{(using a45) and the uniform continuity of $\psi$)}\\
& \sim \sum_{j=n-4}^{n+4} g(u_t)\Delta_j^2 \psi_{r(t)}(a) =
g(u_t)\psi_{r(t)}(a)\\
&\ \ \ \ \ \ \ \ \ \ \ \ \  \ \ \ \ \ \ \ \ \ \ \ \ \ \ \ \ \ \ \ \ \
 \ \ \ \ \ \ \ \ \ \text{(thanks to a44) and the properties of folding data)}\\
&  \sim g(u_t) \overline{\varphi}\left(\lambda_{r(t)}(a)\right) \\
& \ \ \ \ \ \ \ \ \ \ \ \ \  \ \ \ \ \ \ \ \ \ \ \ \ \ \ \ \ \ \ \ \ \
 \ \ \ \ \ \ \ \ \ \text{(thanks to a46))}\\
&\sim
\Phi_t(g \otimes a) \\
& \ \ \ \ \ \ \ \ \ \ \ \ \  \ \ \ \ \ \ \ \ \ \ \ \ \ \ \ \ \ \ \ \ \
 \ \ \ \ \ \ \ \ \ \text{(thanks to a47)).}
\end{split}
\end{equation*}
It follows that $CH\left([\varphi] \star [\lambda]\right)$ is
represented by an asymptotic homomorphism which asymptotically
agrees with $\Phi$.
\end{proof}

\begin{cor}\label{Eth} Let $A',A,B$ be $C^*$-algebras, $A',A$
  separable and $B$ stable and $\sigma$-unital. Assume that $CH :
  \Ext^{-\frac{1}{2}}(A',B) \to [[SA',B]]$ is an isomorphism, and assume
  that $A$ is shape dominated by $A'$. It follows that $CH :
  \Ext^{-\frac{1}{2}}(A,B) \to [[SA,B]]$ is an isomorphism.
\end{cor}
\begin{proof} This follows from a simple diagram chase in the
  commuting diagram
\begin{equation*}\label{G2}
\begin{xymatrix}{
\Ext^{-\frac{1}{2}}(A,B) \ar[rr]^-{CH} \ar@/^1pc/[dd]^-{\mu^*} & &
[[SA,B]] \ar@/^1pc/[dd]^-{(S\mu)^*} \\
& & \\
\Ext^{-\frac{1}{2}}(A',B) \ar@/^1pc/[uu]^-{\lambda^*} \ar[rr]_-{CH} & &
[[SA',B]] \ar@/^1pc/[uu]^-{(S\lambda)^*} }
 \end{xymatrix}
\end{equation*}

\end{proof}

\begin{cor}\label{nuclear} Let $A$ be a separable $C^*$-algebra which is shape
  dominated by a separable nuclear $C^*$-algebra. It follows that $CH :
  \Ext^{-\frac{1}{2}}(A,B) \to [[SA,B]]$ is an isomorphism
  for every stable $\sigma$-unital $C^*$-algebra $B$.
\end{cor}

\begin{cor}\label{homsym} Let $A$ be a separable $C^*$-algebra and $B$
  a stable $\sigma$-unital $C^*$-algebra. Assume that $A$ is homotopy symmetric in the sense of Dadarlat and Loring,
  \cite{DL}. It follows that all extensions of $A$ by $B$ are
  semi-invertible and that $CH :
  \Ext^{-\frac{1}{2}}(A,B) \to [[SA,B]]$ is an isomorphism.
\begin{proof}  The first assertion
  follows from Theorem \ref{THM}. To establish the second we assume
  without loss of generality that $A$ is stable. It follows from \cite{DL} that
  because $A$ is homotopy symmetric the $E$-theory inverse of the
  canonical asymptotic homomorphism $S^2 A \to A$ (arising from the
  Toeplitz extension) has an inverse $A \to S^2A$ giving us a shape
  equivalence between $A$ and $S^2A$. Since $CH : \Ext^{-1/2}(S^2A,B)
  \to [[S^3A,B]]$ is an isomorphism for every stable $\sigma$-unital
  $B$ by \cite{MT3}, it follows from Corollary \ref{Eth} that also $CH :
  \Ext^{-\frac{1}{2}}(A,B) \to [[SA,B]]$ is an isomorphism.
\end{proof}
\end{cor}

It remains an open question if (\ref{CH??}) is always an
isomorphism.

\end{document}